\documentclass[reqno,12pt]{amsart}

\usepackage{amsthm,amsmath,amssymb,mathrsfs,color,url}
\usepackage[pdftex]{graphicx}
\usepackage{tcolorbox}
\usepackage[hang,small,bf]{caption}
\usepackage[subrefformat=parens]{subcaption}
\theoremstyle{plain}
\newtheorem{thm}{Theorem}[section]
\newtheorem{prop}[thm]{Proposition}
\newtheorem{lem}[thm]{Lemma}
\newtheorem{cor}[thm]{Corollary}

\theoremstyle{definition}

\newtheorem{ex}[thm]{Example}

\theoremstyle{remark}

\newtheorem{rem}[thm]{Remark}

\numberwithin{equation}{section}

\def\N{\mathbb{N}}
\def\R{\mathbb{R}}
\def\cB{\mathcal{B}}
\def\cI{\mathcal{I}}
\def\cL{\mathcal{L}}

\def\ss{\mathsf{s}}
\def\deg{\mathrm{deg}}
\def\bra{\langle}
\def\ket{\rangle}
\def\cA{\mathcal{A}}
\def\cB{\mathcal{B}}
\def\fa{\mathfrak{a}}
\def\fb{\mathfrak{b}}
\def\fc{\mathfrak{c}}

\makeatletter
\def\filename{\texttt{\jobname.tex}}
\makeatother
\numberwithin{equation}{section}
\title[Enumeration of connected bipartite graphs with given Betti number]{Enumeration of connected bipartite graphs with given Betti number}
\author{Taro Hasui}
\address{Graduate School of Mathematics, Kyushu University, 744 Motooka, Nishi-ku, Fukuoka 819-0395, Japan}
\email{hasui.taro.826@m.kyushu-u.ac.jp}
\author{Tomoyuki Shirai} 
\address{Institute of Mathematics for Industry,
Kyushu University, 744 Motooka, Nishi-ku, Fukuoka 819-0395, Japan}
\email{shirai@imi.kyushu-u.ac.jp}
\author{Satoshi Yabuoku}
\address{National Institute of Technology, Kitakyushu College, 5-20-1 Shii, Kokuraminamiku, Kitakyushu, Fukuoka 802-0985, Japan}
\email{yabuoku@kct.ac.jp}
\begin{document}
\begin{abstract}
We obtain first order linear partial differential equations which are satisfied by 
exponential generating functions of two variables for the number of connected bipartite graphs with
 given Betti number. By solving these equations inductively,
 we obtain the explicit form of generating functions and 
derive the asymptotic behavior of their coefficients. 
We also introduce a family of basic graphs to classify
 connected bipartite graphs and 
give another expression of the generating functions as the sum over basic graphs
 of rational functions of those for the number of labeled bipartite rooted spanning trees. 
\hfill 
\end{abstract}
\maketitle

\section{Introduction}

Let $G=(V,E)$ be a simple graph, i.e., no self-loops and
multiple edges, and we call it an
$(n,q)$-graph if $|V|=n$ and $|E|=q$. 
We denote the number of connected graphs with $k$ independent cycles, by $N(n,k)$, which is also equal to the number of connected $(n, n-1+k)$-graphs.
Since we are dealing
with connected graphs, we note that $k$ corresponds to the Betti number, the
rank of the first homology group, of each $(n,n-1+k)$-graph. 
Note that $k-1$ is often called \textit{excess} since such
a connected graph has $k-1$ more edges than vertices. 
Connected $(n,n-1)$-graphs are \textit{spanning trees} in
the complete graph $K_n$ over $n$ vertices 
 and it is known as Cayley's formula \cite{C89} that $N(n,0)=n^{n-2}$. 
Connected $(n,n)$-graphs are called \textit{unicycles} and the formula for $N(n,1)$
was found by R\'enyi \cite{R59}, which is given by 
\begin{equation}
 N(n,1) = \frac{1}{2} \left( \frac{h(n)}{n} -
 n^{n-2}(n-1)\right) \sim \sqrt{\frac{\pi}{8}} n^{n-1/2}
 \quad (n \to \infty), 
\label{eq:fnn} 
\end{equation}
where 
\[
 h(n) = \sum_{s=1}^{n-1} {n \choose s} s^{s} (n-s)^{n-s}. 
\]
The asymptotic behavior of $N(n,k)$ for general $k$ as $n \to \infty$ was
also discussed in \cite{W77}, 
where the proofs are based 
on recurrence equations which $N(n,k)$'s satisfy, 
the algebraic structures of generating functions and their
derivatives, and the combinatorial aspect as will be
seen in Theorem~\ref{thm:expressionofF_k} below.

We consider a bipartite simple graph $G = (V,E)$ with 
bipartition $V = V_1 \sqcup V_2$ and call it a bipartite $(r,s,q)$-graph if 
$|V_1|=r$, $|V_2|=s$ and $|E|=q$, which is also considered
as a connected spanning subgraph with $q$-edges in the bipartite graph $K_{r,s}$. 
We denote by $N_{{\rm bi}}(r,s,k)$ the number of connected bipartite
$(r,s,r+s+k-1)$-graphs, whose Betti number is $k$.
Similarly as before, connected bipartite
$(r,s,r+s-1)$-graphs are spanning trees in $K_{r,s}$ and it is well known \cite{Sc62}
that 
\begin{equation}
 N_{{\rm bi}}(r,s,0)=r^{s-1} s^{r-1},   
\label{eq:spanning_trees} 
\end{equation}
which is the bipartite version of Cayley's formula. 

When $rs=0$, we understand $N_{{\rm bi}}(r,s,0) = 1$ if $(r,s)=(1,0), 
(0,1)$; $=0$ otherwise, i.e., the one-vertex simple graph is regarded as
a spanning tree. 
Connected bipartite $(r,s,r+s)$-graphs are unicycles in
$K_{r,s}$ and discussed in the context of cuckoo
hushing by \cite{Ku06}. In the present paper, we discuss
$N_{{\rm bi}}(r,s,k)$ for $k=0,1,\dots$ and the asymptotic
behavior of sum of
$N_{{\rm bi}}(r,s,k)$ with $r+s=n$. 

We consider the exponential generating function of
$N_{{\rm bi}}(r,s,k)$ defined as
follows: for $k=0,1,\dots$, 
\begin{equation}
 F_k(x,y) 
:= \sum_{r,s=0}^{\infty} \frac{N_{{\rm bi}}(r,s,k)}{r!s!}x^r y^s. 
\label{eq:exp_generating_fun}
\end{equation}
For simplicity, we write the exponential generating function
for spanning trees in \eqref{eq:spanning_trees} by 
\begin{equation}
 T(x,y) := F_0(x,y) 
= x+y+\sum_{r,s=1}^{\infty} \frac{r^{s-1}s^{r-1}}{r!s!}x^r y^s.  
\label{eq:Tintro}
\end{equation}
We introduce the following functions of $x$ and $y$: 
\begin{equation}
 T_x = D_x T, \quad T_y = D_y T, 
\quad Z = T_x + T_y, \quad W = T_x T_y, 
\label{eq:TxTyZW}
\end{equation}
where $D_x = x \partial_x$  and $D_y = y \partial_y$ are
the Euler differential operators. 
Then we have the following. 
\begin{prop}\label{prop:Uintro}
The function $F_1(x,y)$ is expressed as $F_1 = f_1(W)$ with 
$f_1(w) = -\frac{1}{2}\big( \log(1-w) + w \big)$, i.e.,  
\[
 F_1(x,y) = -\frac{1}{2}\Big( \log(1-T_xT_y) + T_x T_y \Big).  
\]
\end{prop}
This result was discussed in (cf. \cite[Lemma 4.4]{Ku06},\cite{DK12}). 
However, the term $w$ seems missing in $f_1(w)$ and $F_1(x,y)$ 
was given as $-\frac{1}{2} \log(1-T_xT_y)$, which does not give integer coefficients.  

We will give how to compute $F_k(x,y)$ for general $k$ later and,
in principle, we are able to compute them inductively. Here, we just
give the expression $F_2(x,y)$ (see Remark~\ref{rem:f3f4} for $F_3(x,y)$ and $F_4(x,y)$). 
\begin{thm}\label{thm:W2intro}
The function $F_2(x,y)$ is expressed as $F_2 = f_2(Z,W)$ with 
\begin{equation}
 f_2(z,w) 
= \frac{w^2}{24(1-w)^3} \big\{ (2+3w)z + 2w(6-w) \big\}. 
\label{eq:v}
\end{equation}
\end{thm}

From Proposition~\ref{prop:Uintro} and Theorem~\ref{thm:W2intro}, 
the asymptotic behavior for coefficients of the diagonals $F_1(x,x)$ and $F_2(x,x)$ is derived as follows. 
Let $\bra x^n \ket A(x)$ denote the operation
of extracting the coefficient $a_n$ of $x^n/n!$ in an 
exponential formal power series 
$A(x) = \sum_{n=0}^{\infty} a_n \frac{x^n}{n!}$, i.e. 
\begin{equation}
\bra x^n \ket A(x)
= a_n.  
\label{eq:braket} 
\end{equation}
The coefficients of $\bra x^n \ket F_k(x,x)$ counts the
number of connected bipartite graphs with Betti number $k$
over $n$ vertices, or equivalently, the total number of 
connected bipartite $(r,s,n-1+k)$-graphs with $r+s=n$. 
For convenience, we denote it by
\begin{align}
N_{{\rm bi}}(n, k):=\bra x^n \ket F_k(x,x)
=\sum_{r+s=n}\binom{n}{r}N_{{\rm bi}}(r,s,k).
\label{eq:defN_{bi}(n,k)}
\end{align}
When $k=0$, we have 
\begin{align*}
 F_0(x,x) 
= 2 \Big(x + \sum_{n=2}^{\infty} n^{n-2} \frac{x^n}{n!}
 \Big), 
\end{align*}
hence $N_{{\rm bi}}(n, 0)=2n^{n-2}=2N(n,0)$, 
which is equivalent to \eqref{eq1}. That is, as we will see in Section~\ref{sec:asymptotics}, 
the spanning trees in $K_{r,s}$ for some $(r,s)$ with $r+s=n$ 
are in two-to-one correspondence with those in $K_n$.  
When $k \ge 1$, the situation is different since there may exist cycles having odd length in $K_n$ while cycles must have even length in $K_{r,s}$.  
From Proposition~\ref{prop:Uintro} and
Theorem~\ref{thm:W2intro}, we obtain the asymptotic behavior
of $N_{{\rm bi}}(n,1)$ and $N_{{\rm bi}}(n, 2)$.
\begin{thm}\label{thm:asympofun} 
For $n=4,5,\dots$, 
\begin{align*}
N_{{\rm bi}}(n, 1)
=n^{n-1} \sum_{2 \le k \le n/2}\dfrac{n !}{(n-2k)! n^{2k}}
\sim \sqrt{\dfrac{\pi}{8}} n^{n-1/2} 
 \quad (n \to \infty). 
\end{align*}
\end{thm}
\begin{figure}[htbp]
{\footnotesize 
\begin{center}
\begin{tabular}{|c|ccccccccc|} \hline
$n$ & 3 & 4 & 5 & 6 & 7 & 8 & 9 & 10 & 11 \\\hline 
$N(n,1)$ & 1 & 15 & 222 & 3660 & 68295 & 1436568 & 33779340 & 880107840& 25201854045\\
$N_{{\rm bi}}(n, 1)$  & 0 & 6 & 120& 2280 & 46200 & 1026840 & 25102224& 673706880& 19745850960\\
\hline
\end{tabular}
\end{center}
}
\caption{$N(n,1)$ and $N_{{\rm bi}}(n, 1)$ for $n=3,4,\dots,11$} 
\end{figure}

From \eqref{eq:fnn}, this shows that the main term of the asymptotic behavior of the number of
bipartite unicycles over $n$ vertices is the same 
as that of the number of unicycles. 

\begin{thm}\label{thm:asympofF_2} 
As $n\to\infty$,
\begin{equation}
N_{{\rm bi}}(n, 2)
 \sim \frac{5}{48}n^{n+1}.
\label{eq:f2xx} 
\end{equation}
\end{thm}
It is known \cite{W77} that 
in the case of $K_n$, the main term of
asymptotic behavior of the number of ``bicycles'' is known to be $\frac{5}{24}n^{n+1}$, which is
twice of \eqref{eq:f2xx}. 
\begin{figure}[htbp]
{\footnotesize 
\begin{center}
\begin{tabular}{|c|cccccccc|} \hline
$n$ & 4 & 5 & 6 & 7 & 8 & 9 & 10 & 11 \\\hline 
$N(n,2)$ & 6 & 205 & 5700 & 156555 & 4483360 & 136368414 &
			     4432075200 & 154060613850\\
$N_{{\rm bi}}(n, 2)$ & 0 & 20 & 960 & 33600 & 1111040 & 37202760 & 1295884800 &
47478243120 \\ \hline
\end{tabular}
\end{center}
}
\caption{$N(n,2)$ and $N_{{\rm bi}}(n, 2)$ for $n=4,5,\dots,11$} 
\end{figure}

For general $k$, we have the following asymptotic
equality. 
\begin{thm}\label{thm:generalfk}
For $k \ge 0$, as $n \to \infty$, 
\begin{equation}\label{eq:f(n,n+k)asmthm}
N_{{\rm bi}}(n, k)
 \sim \frac{1}{2^{k-1}} 
N(n,k).
\end{equation}
\end{thm}
The proof of \eqref{eq:f(n,n+k)asmthm} is given in
Section~\ref{sec:asymp_equality}.
The following asymptotic behavior
\[ 
 N(n,k) \sim \rho_{k-1} n^{n+(3k-4)/2} \quad (n \to \infty)
\]
is given in \cite{W77}, where the 
explicit value of $\rho_k$ can be computed by the recurrence
equation. Comparing the generating function of \cite[Section 8]{W77} with that of this paper, we can see that the subscript $k$ is off by one. However, the meaning of both is the same.
To derive \eqref{eq:f(n,n+k)asmthm}, we use the following result, which would be interesting on its own right and give more detailed information.   

\begin{thm}\label{thm:expressionofF_k}
For $k \ge 2$, $F_k(x,y)$ is decomposed into the sum of
 rational functions of $T_x$ and $T_y$ over the set $BG_k$
 of basic graphs with Betti number $k$ as 
\begin{align}\label{eq:expressionofF_k}
F_k(x,y)=\sum_{\cB \in BG_k} J_{\cB}(x,y)
\end{align}
with 
\begin{equation}
 J_{\cB}(x,y) = \frac{ T_x^{|V_1|}T_y^{|V_2|}}{g_{\cB}
  (1-T_xT_y)^{N_{\rm sp} + k-1-e}}
\label{eq:J_B}
\end{equation}
where $V_1 \sqcup V_2$ is the vertex set of $\cB$, 
$g_{\cB}$ is the number of automorphisms of $\cB$, 
$N_{\rm sp}$ and $e$ are the numbers of vertices with
degree $\ge 3$ and $\delta$-edges in $\cB$, respectively. 
\end{thm}

The definitions of basic graph and $\delta$-edge will
be given in the proof of Theorem~\ref{thm:expressionofF_k}.
From this theorem, we conclude at least that $F_k(x,y)$ for $k \ge 2$ is a rational function of $T_x$ and $T_y$. 
Note that $F_k(x,y)$ is symmetric with respect to $T_x$ and $T_y$ by the bipartite structure, and $F_k(x,y)$ can be expressed in $Z$ and $W$ for $k \geq 2$ as follows.
\begin{thm}\label{thm:intro}
For $k \ge 2$, the function $F_k(x,y)$ is expressed as $F_k = f_k(Z,W)$ with
\begin{equation}
 f_k(z,w) 
= \frac{w^2}{(1-w)^{3(k-1)}} \sum_{j=0}^{k-1} q_{k,j}(w)z^j,  
\label{eq:generalfk}
\end{equation}
where $q_{k,j}(w)$ is a polynomial in $w$.
\end{thm}
For $k \geq 3$, the generating function becomes highly complicated (see Remark~\ref{rem:f3f4}) 
and although we can write it down explicitly in principle as in Theorem~\ref{thm:W2intro}, 
it may not  be practical to do so, instead, we here emphasize that the generating function has a particular form given by \eqref{eq:generalfk}. 
The polynomial $q_{k,j}(w)$ seems to have more factor $w^{b_{k,j}}$ depending on $k$ and $j$.  

The paper is organized as follows. In Section 2, we give
recurrence equations for $N_{{\rm bi}}(r,s,q-r-s+1)$ 
and derive recurrence linear partial differential equations that the
generating functions $F_k(x,y)$ of $N_{{\rm bi}}(r,s,k)$
satisfy. In Section 3, we solve these equations by reducing
them to a system of ordinary differential equations and
obtain the explicit expressions of $F_1(x,y)$ and
$F_2(x,y)$. In Section 4, we obtain the asymptotic
behavior of the coefficients of $F_k(x,x)$ for $k=1,2$. 
In Section 5, we will give proofs of Theorem~\ref{thm:expressionofF_k} and Theorem~\ref{thm:intro} and another proof of Proposition \ref{prop:Uintro} by a combinatorial argument. 
In Section 6, we will give proof of Theorem~\ref{thm:generalfk}.

\section{Recurrence equations}
\label{sec:recurrence}

Let $N_{{\rm bi}}(r,s,q-r-s+1)$ be the number of connected bipartite
$(r,s,q)$-graphs as defined in the introduction. 
Since an $(r,s,r+s-1)$-bipartite graph is a spanning tree and
we are dealing with simple graphs, 
it is clear that 
\begin{equation}
N_{{\rm bi}}(r,s,q-r-s+1) =0 \quad \text{if $q< r+s-1$ or $q > rs$}.   
\label{eq:fequalto0}
\end{equation}
As mentioned in \eqref{eq:spanning_trees}, $N_{{\rm bi}}(r,s,0) = r^{s-1} s^{r-1}$. 
Here we understand $0^a = \delta_{0,a}$ as Kronecker's
delta. For example, $N_{{\rm bi}}(1,0,0)=N_{{\rm bi}}(0,1,0)=1$ and $N_{{\rm bi}}(0,0,0)=0$. 

\begin{lem}\label{lem:recurrence1}
For $(r,s) \not= (0,0)$ and $q=-1,0,1,\dots$, we have the following recurrence
 equations: 
\begin{equation}
 (q+1) N_{{\rm bi}}(r,s,q-r-s+2) = (rs -q) N_{{\rm bi}}(r,s,q-r-s+1) + Q(r,s,q), 
\label{eq:recurrence1} 
\end{equation}
where 
\begin{align}
Q(r,s,q) 
&= 
\frac{1}{2} \sum_{r_1=0}^r \sum_{s_1=0}^s \sum_{t=0}^q 
{r \choose r_1} {s \choose s_1} 
\{(r-r_1)s_1 + r_1(s-s_1)\} \nonumber \\
&\quad \times 
N_{{\rm bi}}(r_1,s_1,t-r_1-s_1+1) N_{{\rm bi}}(r-r_1,s-s_1,q-t-(r-r_1)-(s-s_1)+1) 
\label{eq:recurrence1-2}
\end{align}
and $Q(r,s,-1)=0$.
\end{lem}

\begin{proof}
Here we give a sketch of the proof. 
Let $G= (V_1,V_2,E)$ be an $(r,s)$-bipartite graph with $q$ edges
and we add an edge to make a connected $(r,s)$-bipartite
 graph with $q+1$ edges. 
There are two cases: (i) $G$ itself is connected and (ii) $G$
 consists of two connected bipartite components. 
For the case (i), we add an edge joining $V_1$ and $V_2$. 
For the case (ii), if $V_j = V_{j,1} \sqcup V_{j,2} \
 (j=1,2)$, then there are four ways to add an edge joining
 two bipartitions, i.e., 
$V_{1,1}$ and $V_{2,1}$,  
$V_{1,1}$ and $V_{2,2}$,  
$V_{1,2}$ and $V_{2,1}$, or 
$V_{1,2}$ and $V_{2,2}$. 
\end{proof}

From Lemma~\ref{lem:recurrence1}, 
we have the following recurrence linear partial differential
equations for generating functions $\{F_k\}_{k=0,1,\dots}$ defined by \eqref{eq:exp_generating_fun}. 
For the sake of convenience, we also consider $F_{-1}$, which 
is equal to $0$ from \eqref{eq:fequalto0}. 

\begin{prop}\label{prop:recurrence} 
For $k=-1,0,1,2,\dots$, 
\begin{align}
\lefteqn{(D_x+D_y+k)F_{k+1}} \nonumber \\ 
&= (D_xD_y- D_x-D_y+1-k)F_k + \sum_{l=0}^{k+1}
D_x F_l \cdot D_y F_{k+1-l}, 
\label{eq:PDE}
\end{align}
where $D_x = x \partial_x$ and $D_y = y \partial_y$. 
\end{prop}
\begin{proof}
From Lemma~\ref{lem:recurrence1}, we have 
\begin{align}
\lefteqn{(r+s+k) N_{{\rm bi}}(r,s,k+1)} \nonumber\\
&= (rs -r-s+1-k) N_{{\rm bi}}(r,s,k) + Q(r,s,r+s-1+k), 
\label{eq:recurrence1} 
\end{align}
where 
\begin{align}
Q(r,s,r+s-1+k) 
&= 
\frac{1}{2} \sum_{r_1=0}^r 
\sum_{s_1=0}^s \sum_{t=0}^{k+1} 
{r \choose r_1} {s \choose s_1} 
\{(r-r_1)s_1 + r_1(s-s_1)\} \nonumber \\
&\quad \times 
N_{{\rm bi}}(r_1,s_1,t) N_{{\rm bi}}(r-r_1,s-s_1,k-t+1). 
\end{align}
Here we used the fact that 
$N_{{\rm bi}}(r_1,s_1,t) N_{{\rm bi}}(r-r_1,s-s_1,k-t+1)=0$ 
unless 
$t+r_1+s_1-1 \geq r_1+s_1-1$ and $t+r_1+s_1-1 \geq r_1+s_1+k$, i.e., $0 \leq t \leq k+1$.  

By multiplying both sides of \eqref{eq:recurrence1} and
 taking sum over $r,s=0$ to $\infty$, we see that 
\begin{align*}
(D_x+D_y+k)F_{k+1}
&= (D_xD_y- D_x-D_y+1-k)F_k \\
&\quad + \frac{1}{2} \sum_{l=0}^{k+1}
\{D_x F_l \cdot D_y F_{k+1-l} 
+ D_y F_l \cdot D_x F_{k+1-l} \} \\
&= (D_xD_y- D_x-D_y+1-k)F_k + \sum_{l=0}^{k+1}
D_x F_l \cdot D_y F_{k+1-l}.  
\end{align*}
\end{proof}

In what follows, we write $T:=F_0$ and use the symbols $T_x,
T_y, Z, W$ in \eqref{eq:TxTyZW}. 
We think of $T$ as a known function below. 
These functions satisfy several useful identities. 

First let us consider the case $k=-1$ in \eqref{eq:PDE}. Then we have 
\begin{align*}
(D_x+D_y-1)F_0 = D_x F_0 \cdot D_y F_0,  
\end{align*}
which is equivalent to 
\begin{equation}
 T_x T_y = T_x+T_y -T.  
\label{eq:T-formula}
\end{equation}

\begin{rem}
As in the above, in Sections 2 and 3, we always use the
subscript $x,y$, etc. for the differentiation by 
Euler operators $D_x=x\partial_x, D_y = y\partial_y$,
etc., but not the usual partial derivative $\partial_x, \partial_y$, etc. 
\end{rem}

For $k=1,2,\dots$, from \eqref{eq:PDE}, we have the following linear PDE 
\begin{equation}
\cL_k  F_{k+1}
= (D_xD_y- D_x-D_y+1-k)F_k + \sum_{l=1}^{k}
D_x F_l \cdot D_y F_{k+1-l}, 
\label{eq:PDE2}
\end{equation}
where  
\begin{equation}
\cL_k := (1-T_y)D_x+(1-T_x)D_y+k. 
\label{eq:Lk} 
\end{equation}
Therefore, in principle, we can solve the
\eqref{eq:PDE2} recursively and 
obtain $F_k$ for $k=1,2,\dots$ in terms of the known function $T$. 
Before solving these equations, we observe several algebraic relations for $T_x$'s. 

\begin{lem} The following identities hold. 
\begin{align}
T_{xx} &= T_x (T_{xy}+1) \label{eq:31}\\
T_{xy} &= T_{yx} = T_x T_{yy} = T_y T_{xx} \label{eq:32}\\
T_{yy} &= T_y (T_{xy}+1). \label{eq:33}
\end{align}
Furthermore, 
\begin{equation}
T_{xy} = \frac{T_xT_y}{1-T_xT_y}. 
\label{eq:34}
\end{equation}
\end{lem}
\begin{proof}
It is known that two functions $T_x$ and $T_y$ satisfy the
 following functional equations (cf. \cite[Section 3]{Ku06}):  
\begin{equation}
T_x = x e^{T_y}, \quad 
T_y = y e^{T_x}. 
\label{eq:0}
\end{equation}
Differentiating both sides of \eqref{eq:0} yields the identities \eqref{eq:31}, 
\eqref{eq:32}, and \eqref{eq:33}. 
Plugging \eqref{eq:31} into \eqref{eq:32} yields \eqref{eq:34}. 
\end{proof}

By using the notations \eqref{eq:TxTyZW}, 
we can rewrite \eqref{eq:T-formula} and \eqref{eq:34} as 
\begin{equation}
 T = Z-W
\label{eq:T}
\end{equation}
and 
\begin{equation}
 T_{xy} = \frac{W}{1-W}, 
\label{eq:Txyw}
\end{equation}
respectively. 

Functions of $Z$ and $W$ are well-behaved under the action of $\cL_k$. 
\begin{lem}\label{lem:L0}
Suppose $F(x,y)$ and $G(x,y)$ admit differentiable
 functions $f(z)$ and $g(w)$ such that 
$F(x,y) = f(Z(x,y))$ and $G(x,y) = g(W(x,y))$, respectively. 
Then,  
\begin{align}
 \cL_0 F &= (D_z f)(Z), \label{eq:L0f}\\
 \cL_0 G &= 2 (D_w g)(W), \label{eq:L0g}
\end{align}
where $(D_u f)(u) = u f'(u)$. 
Moreover, $H = h(Z,W)$ for a differentiable function $h(z,w)$,  
\begin{equation}
\cL_0 H = (D_z h)(Z,W)
+ 2(D_w h)(Z,W).  
\label{eq:L0-2}
\end{equation}
\end{lem}
\begin{proof}
From \eqref{eq:31}, we have 
\begin{align*}
 D_x Z &= T_{xx} + T_{yx}
= T_x + (T_x + 1) T_{xy}, \\
 D_y Z &= T_{xy} + T_{yy} = T_y + (T_y + 1) T_{xy}. 
\end{align*}
From \eqref{eq:34}, we see that 
\begin{align*}
 \cL_0 Z 
&= (1-T_y)\{T_x + (T_x + 1) T_{xy}\} 
+ (1-T_x)\{T_y + (T_y + 1) T_{xy}\} \\
&= Z - 2W + 
\{(1-T_y)(T_x + 1) + (1-T_x) (T_y + 1)\} T_{xy} \\
&= Z.  
\end{align*}
In general, since $\cL_0$ is a linear operator, we see that 
\[
 \cL_0 f(Z) 
= f'(Z) \cL_0 Z 
= Z f'(Z)
= (Df)(Z). 
\] 
We note that from the definition of $\cL_0$, 
\[
 \cL_0 T = (1-T_y) T_x + (1-T_x) T_y = Z - 2W. 
\]
Since $W = Z - T$ from \eqref{eq:T}, we have 
\[
\cL_0 W = \cL_0 Z - \cL_0 T = 2W. 
\]
Therefore, 
\[
 \cL_0 g(W) 
= g'(W) \cL_0 W 
= 2W g'(W) 
= 2 (Dg)(W). 
\] 
For general $h(Z,W)$, we obtain \eqref{eq:L0-2} similarly.
This completes the proof. 
\end{proof}

From this formula, we can reduce the analysis on $F(x,y) =
h(Z(x,y), W(x,y))$ to that on $h(z,w)$ of two variables $z$
and $w$.

\section{Explicit expressions of generating functions}

In this section, we solve the PDE \eqref{eq:PDE2} to obtain 
the explicit expressions of generating functions $F_1$ and 
$F_2$. The algebraic relations of $Z, W$ and their derivatives, 
which were seen in the previous section, play an essential role of the proof. 

\subsection{For $F_1$: unicycles}\label{sec:F1}
For unicycles, we will solve \eqref{eq:PDE2} with $k=0$, i.e., 
\begin{equation}
\cL_0  F_1
= (D_xD_y- D_x-D_y+1) F_0. 
\label{eq:PDE2-1}
\end{equation}
By using $T$ and their derivatives, 
we can rewrite 
\eqref{eq:PDE2-1} as 
\begin{equation}
\cL_0 F_1 
= T_{xy} - T_x - T_y + T. 
\label{eq:PDE2-1-2}
\end{equation}
The right-hand side is a function of $W$ and is written 
\[
T_{xy} - T_x - T_y + T = \frac{W}{1-W} - W,  
\]
from which together with \eqref{eq:L0g} we see that $U$ is also a function of
$W$ and obtain the following. 

\begin{proof}[Proof of Proposition~\ref{prop:Uintro}]
Suppose there exists a function $f_1=f_1(w)$ such that $F_1 = f_1(W)$. 
By definition, $f_1$ does not have a constant term, i.e., $f_1(0)=0$. 
Since $\cL_0 F_1 = 2 (Df_1)$ by
\eqref{eq:L0g}, \eqref{eq:PDE2-1-2} can be
expressed as 
\[
 2(Df_1)(w) = \frac{w}{1-w} - w, 
\]
or equivalently, 
\[
 f_1'(w) = \frac{1}{2} \left(\frac{1}{1-w} - 1 \right). 
\]
From this differential equation with $f_1(0)=0$, we obtain 
\[
 f_1(w) = -\frac{1}{2}\big( \log(1-w) + w \big) 
\]
and thus we obtain the assertion. 
\end{proof}

\subsection{For $F_2$: bicycles}\label{sec:F2}

We want to solve \eqref{eq:PDE2} with $k=1$, i.e., 
\begin{equation}
\cL_1  F_2
= (D_xD_y- D_x-D_y)F_1 + D_x F_1 \cdot D_y F_1
\label{eq:PDE2-2}
\end{equation}
where $\cL_1 = \cL_0 + 1$. 
Here $F_1$ has been given in already given in
Proposition~\ref{prop:Uintro} and considered as a known function. 
We will solve this equation to prove Theorem~\ref{thm:W2intro}.  

Before proceeding to the proof, we prepare some lemmas. 
\begin{lem}
\begin{align*}
Z_x = \frac{W+T_x}{1-W}, \quad Z_y = \frac{W+T_y}{1-W},
 \quad Z_{xy} = \frac{2+Z}{(1-W)^2} T_{xy}. 
\end{align*}
Moreover, 
\begin{equation}
 Z_x + Z_y = \frac{Z+2W}{1-W}, \quad  
 Z_x Z_y = \frac{W(Z+W+1)}{(1-W)^2}.
\label{eq:zxzy}
\end{equation}
\end{lem}

\begin{lem}
\begin{equation}
 W_x = (1+T_x)T_{xy}, \quad 
 W_y = (1+T_y)T_{xy} 
\label{eq:wxwy}
\end{equation}
and 
\begin{equation}
 W_{xy} = T_{xy}^2  + \frac{1+Z+W}{(1-W)^2} T_{xy}.
\label{eq:wxy}
\end{equation}
Moreover, 
\begin{equation}
 W_x + W_y = \frac{W(Z+2)}{1-W}, \quad  
 W_x W_y = \frac{W^2(Z+W+1)}{(1-W)^2}
\label{eq:wxwy2}
\end{equation}
and 
\begin{equation}
 Z_x W_y + Z_y W_x = \frac{W(ZW+Z+4W)}{(1-W)^2}. 
\label{eq:zxwy}
\end{equation}
\end{lem}
\begin{proof}
First it follows from \eqref{eq:32} that 
\[
 W_x = (T_x T_y)_{x} = T_{xx} T_y + T_x T_{xy} 
= (1+T_x)T_{xy}. 
\]
By symmetry, we have the second equation in
 \eqref{eq:wxwy}. 
Next it follows from \eqref{eq:34} that 
\begin{equation}
 (T_{xy})_y = \left(\frac{W}{1-W}\right)_y =
 \frac{1}{(1-W)^2} W_y
=  \frac{1}{(1-W)^2} (1+T_y) T_{xy}
\label{eq:txyy} 
\end{equation}
Then, 
\[
 W_{xy} = ((1+T_x)T_{xy})_y = T_{xy}^2  + (T_x+1) (T_{xy})_y
= T_{xy}^2  + (1+T_x)(1+T_y)
\frac{1}{(1-W)^2} T_{xy}. 
\]
\end{proof}

Later we will also use the following. 

\begin{proof}[Proof of Theorem~\ref{thm:W2intro}]
Let $G=-2F_1$, i.e., 
\[
 G(W) = \log (1-W) + W.  
\]
First we observe that 
\[
 G_x = \left(\frac{-1}{1-W}+1\right)W_x 
= \frac{-W}{1-W}W_x
= -T_{xy}W_x. 
\]
Similarly, $G_y = -T_{xy} W_y$. Hence, 
\[
 G_x + G_y = -T_{xy} (W_x + W_y) = -(2+Z)T_{xy}^2. 
\]
Next it follows from \eqref{eq:wxwy}, 
\eqref{eq:wxy} and \eqref{eq:txyy} that 
\begin{align*}
 G_{xy} 
&= -(T_{xy}W_x)_y \\
&= -(T_{xy})_y W_x - T_{xy} W_{xy}  \\
&= -\frac{1}{(1-W)^2} (1+T_y) T_{xy} \cdot (1+T_x)T_{xy}
- T_{xy} \left(T_{xy}^2  + (1+T_x)(1+T_y)
\frac{1}{(1-W)^2} T_{xy}\right)
 \\
&=
-T_{xy}^2 \left\{
\frac{2}{(1-W)^2}(1+T_x)(1+T_y)+ T_{xy}
\right\} \\
&=
-T_{xy}^2 \left\{
\frac{2}{(1-W)^2}(1+Z+W)+ T_{xy}
\right\}.  
\end{align*}
Lastly, we have 
\begin{align*}
G_x G_y 
&= T_{xy}^2 W_x W_y 
= T_{xy}^4 (1+T_x)(1+T_y)
\end{align*}
Putting the above all together in \eqref{eq:PDE2-2}, we have 
\begin{align}
4\cL_1 F_2 
&= 2 (G_x+G_y) + G_x G_y - 2 G_{xy} \nonumber\\
&= -2(Z+2)T_{xy}^2 + T_{xy}^4 (1+T_x)(1+T_y) + 2 
T_{xy}^2 \left\{
\frac{2}{(1-W)^2} (1+T_x)(1+T_y)
+ T_{xy} 
\right\} \nonumber\\
&= \frac{T_{xy}^2}{(1-W)^2}
\left\{
-2(Z+2)(1-W)^2 + W^2 (1+Z+W) + 
4 (1+Z+W) + 2W(1-W)
\right\} \nonumber\\
&= \frac{W^2}{(1-W)^4}
\big\{
(-W^2+4 W+2) Z + (W^2-5 W+14) W
\big\}. 
\label{eq:forF2}
\end{align}
Suppose there exist functions $a_0(w)$ and $a_1(w)$ 
such that $F_2 = f_2(Z,W)$ with $f_2(z,w) := a_1(w)z+a_0(w)$. 
Since $\cL_1 = \cL_0 + 1$, from \eqref{eq:L0-2}, 
the \eqref{eq:forF2} can be
expressed as 
\begin{equation}
4(D_z f_2 + 2 D_w f_2 + f_2)
= 
\frac{w^2}{(1-w)^4}
\{
(-w^2+4 w+2) z + (w^2-5 w+14) w
\}. 
\label{eq:ab1}
\end{equation}
On the other hand, since $f_2(z,w) = a_1(w)z+a_0(w)$, we have 
\begin{align}
D_z f_2 + 2 D_w f_2 + f_2
&= a_1(w)z + 2 \{(Da_1)(w) z + (Da_0)(w)\} + a_1(w)z + a_0(w)
 \nonumber \\
&= \{2a_1(w)+2 (Da_1)(w)\} z + \{2 (Da_0)(w) + a_0(w)\}. 
\label{eq:ab2}
\end{align}
Comparing \eqref{eq:ab1} with \eqref{eq:ab2} yields 
\[
a_0(w) +  2(Da_0)(w) = 
\frac{w^3}{4(1-w)^4}
(w^2-5 w+14)
\]
and 
\[
 2a_1(w)+2 (Da_1)(w) = 
\frac{w^2}{4(1-w)^4}(-w^2+4 w+2).  
\]
On the other hand, by the definition of $F_2(x,y)$, the function $f_2(z,w)$ does not have the terms $z^i, i=0,1,2,\dots$ since if such a term appears in $f_2(z,w)$, so do the terms $x^i$ and $y^i$ in $F_2(x,y)$, which contradicts to the fact that $N_{\rm{bi}}(i,0,2)=N_{\rm{bi}}(0,i,2)=0$. This implies that $a_0(0) = a_1(0)=0$. 
Then, we can easily solve the above differential equations with initial conditions $a_0(0) = a_1(0)=0$ to obtain
\[
a_0(w) = \frac{w^3(6-w)}{12(1-w)^3}, \quad 
a_1(w) = \frac{w^2(2+3w)}{24(1-w)^3}. 
\]
Therefore, 
\[
 f_2(z,w) = \frac{w^2(2+3w)}{24(1-w)^3} z +
 \frac{w^3(6-w)}{12(1-w)^3}. 
\]
This completes the proof. 
\end{proof}

\begin{rem}\label{rem:f3f4}
We can continue the above computations for $F_k(x,y) = f_k(z,w)$. 
Here we give $f_3(z,w)$ and $f_4(z,w)$ just for the reference:  
\begin{align*}
 f_3(z,w)
&= \frac{w^3(5+41w-23w^2+8w^3-w^4)}{24(1-w)^6} + \frac{w^3(32+34w-9w^2+3w^3)}{48(1-w)^6} z \\
&\quad + \frac{w^2(1+8w+6w^2)}{48(1-w)^6} z^2 
\end{align*}
and 
\begin{align*}
 f_4(z,w)
&=\frac{w^3\left(-76 w^7+809 w^6-3746 w^5+9889 w^4-15356
 w^3+22820 w^2+7680 w+80\right) }{2880 (1-w)^9} \\
&+\frac{w^3\left(230 w^6-1425 w^5+5568 w^4-6617
   w^3+30468 w^2+35988 w+2088\right) }{5760
 (1-w)^9} z \\
&+\frac{w^3\left(61 w^4+64 w^3+1186 w^2+1692 w+312\right) 
}{576 (1-w)^9} z^2 \\
&+\frac{w^2\left(254 w^4+1919
   w^3+2624 w^2+704 w+24\right) }{5760 (1-w)^9} z^3. 
\end{align*}
These expressions lead to \eqref{eq:generalfk} in Theorem~\ref{thm:generalfk}. 
\end{rem}

\section{Asymptotic behaviors of the coefficients}
\label{sec:asymptotics}

\subsection{Asymptotic behavior of the coefficients of $F_1(x,x)$}

We use the notation \eqref{eq:braket}. 
We recall the convolution of exponential generating
functions 
\begin{equation}
\bra x^n \ket A(x)B(x) = \sum_{k=0}^n {n \choose k} a_k
b_{n-k} 
\label{eq:convolution} 
\end{equation}
when $\bra x^n \ket A(x) = a_n$ and $\bra x^n \ket B(x) =
b_n$. 
For an exponential power series $C(x,y) = \sum_{r,s=0}^{\infty} c_{rs}
 \frac{x^ry^s}{r!s!}$ of two variables, we use the notation 
\[
 \bra x^r y^s \ket C(x,y) = c_{rs},  
\]
and we note that the coefficients of the diagonal
$C(x,x)$ is given by 
\[
 \bra x^n \ket C(x,x) = \sum_{r+s=n} {n \choose r}c_{rs}.  
\]
In Section~\ref{sec:F1}, we derived the generating function
$F_1(x,y)$ for unicycles. 
In this section, we focus on the coefficients of the
diagonal $F_1(x,x)$, 
\[
 N_{\rm{bi}}(n,1):= \bra x^n \ket F_1(x,x) = 
\sum_{r+s=n} {n \choose r} N_{{\rm bi}}(r,s,1), 
\]
which corresponds to the total of the numbers of complete
unicycles over $n$ vertices. 
We will see the asymptotic behavior of $N_{\rm{bi}}(n,1)$ as $n \to
\infty$. 

From Proposition~\ref{prop:Uintro}, we have
\begin{equation}
F_1(x,x)=\frac{1}{2}\sum_{k=2}^{\infty}\frac{W(x,x)^k}{k}. 
\label{eq:coeffofUxx}
\end{equation}
First we consider the coefficients of the diagonal $W(x,x)$. 
Since $W=T_x+T_y-T$ from \eqref{eq:T}, it is easy to see that 
\[
W(x,y) = \sum_{r,s =1}^{\infty} \frac{w(r,s)}{r!s!}x^r y^s,  
\]
where $w(r,s) = r^{s-1}s^{r-1}(r+s-1)$. 
Hence, we have 
\[
W(x,x) =\sum_{n=2}^{\infty} 
\left(\sum_{r+s=n}\frac{w(r,s) n!}{r!s!}\right)\frac{x^n}{n!}
=: \sum_{n =2}^{\infty} w_n \frac{x^n}{n!}, 
\]
where 
\begin{align}\label{w*1}
w_n & =\sum_{r+s=n}\frac{r^{s-1}s^{r-1}(r+s-1) n!}{r!s!}\nonumber\\
&=(n-1)\sum_{r=1}^{n-1}\binom{n}{r}r^{n-r-1}(n-r)^{r-1}.
\end{align}
The sum in \eqref{w*1} can be computed by the following
identity (cf. \cite{KP09}). 
\begin{lem}\label{keylem1} For $n=2,3,\dots$,  
\begin{align}\label{eq1}
\sum_{r=1}^{n-1}\binom{n}{r}r^{n-r-1}(n-r)^{r-1}=2n^{n-2}.
\end{align}
\end{lem}
\begin{proof}
Here we give a combinatorial proof of the identity. 
Let $S_{n}$ and $S_{n}^b$ be the set of labeled
 spanning trees on $K_{n}$ and that of
 labeled spanning trees on the complete bipartite graph with
 $n$ vertices, respectively. Also, let $S_{r,s}^b$ be the
 set of labeled spanning trees on the complete bipartite
 graph $K_{r,s}$. Then
\begin{align*}
S_{n}^b = 
\bigsqcup_{1\le r \le n-1} S_{r,n-r}^b.
\end{align*} 
For $(V_1\sqcup V_2, E_{r,n-r}) \in S_{r, n-r}^b$ with
 $|V_1|=r$ and $|V_2|=n-r$, 
we define a map $\eta: S_{n}^b \to S_{n}$ by 
\begin{align*}
\eta((V_1\sqcup V_2, E_{r,n-r})):=(V, E_{r,n-r}), 
\end{align*} 
i.e., the map of forgetting partitions. 
Since every spanning tree on $K_{n}$ is bipartite, $\eta$
 is surjective. Moreover, $\eta$ is two-to-one
 mapping. Indeed, for $1\le r \le n-1$ and $(V_1\sqcup
 V_2, E_{r,n-r}) \in S_{r, n-r}^b$, there exists a
 unique spanning tree $(V_1'\sqcup V_2', E_{n-r,r}) \in
 S_{n-r,r}^b$ such that $V_1'=V_2, V_2'=V_1$ and
 $E_{n-r,r}=\{(i,j) \in V_1'\times V_2' : (j,i) \in
 E_{r,n-r} \}$. Now we derive \eqref{eq1}. 
For $1\le r \le n-1$, $|S_{r,n-r}^b|
 =\binom{n}{r}N_{{\rm bi}}(r,n-r,0)=\binom{n}{r}r^{n-r-1}(n-r)^{r-1}$
 by the choice of labeled $r$ vertices in $V_1$ and 
\eqref{eq:spanning_trees}.  
Hence
\begin{align*}
|S_{n}^b| = \sum_{r=1}^{n-1}\binom{n}{r}r^{n-r-1}(n-r)^{r-1}.
\end{align*}
On the other hand, $|S_{n}| = n^{n-2}$ by Cayley's formula. 
Therefore, we conclude that \eqref{eq1} holds from the two-to-one correspondence of $\eta$. 
\end{proof}

\begin{cor}\label{cor:wn} 
For $n=1,2,\dots$, 
\begin{equation}
w_n = \bra x^n \ket W(x,x) = 2(n-1) n^{n-2}. 
\label{eq:w1} 
\end{equation}
\end{cor}

Now we proceed to the case of the power of $W(x,x)$. 
For $k=1,2,\dots$, we write 
\[
 w_n^{*k} := \bra x^n \ket W(x,x)^k. 
\]
In particular, $w_n^{*1} = w_n$ in Corollary~\ref{cor:wn}. 
Note that the smallest degree of the terms in $W(x,x)$ is 2
and hence $w_n^{*k} = 0$ for $n=1,2,\dots,2k-1$. 
From \eqref{eq:convolution}, $w_n^{*k}$ is the $k$-fold convolution of
$(w_n)_{n=2,3,\dots}$ and inductively defined by 
\begin{equation}
w_n^{*(k+1)}
=\sum_{r=2k}^{n-2}\binom{n}{r}w_r^{*k} w_{n-r}. 
\label{eq:repeated-conv}
\end{equation}
From \eqref{eq:coeffofUxx}, the coefficients $N_{\rm{bi}}(n,1)$ of
$F_1(x,x)$ are given by
\begin{align}\label{coefu}
N_{\rm{bi}}(n,1)
=\frac{1}{2}\sum_{2 \le k \le n/2} \frac{w_n^{*k}}{k}. 
\end{align}

\begin{prop}\label{prop:wnconv}
For $k=1,2,\dots, \lfloor n/2 \rfloor$,
\begin{align}
 w_n^{*k} = 2k \cdot (2k)! n^{n-2k-1} {n \choose 2k}. 
\label{eq:conv}
 \end{align}
\end{prop}
\begin{proof}
For fixed $n$, we prove the \eqref{eq:conv} by induction in $k$. 
For $k=1$, it is obviously true since $w_n^{*1} = w_n$. 
Suppose that \eqref{eq:conv} holds for up to $k$, then by \eqref{eq:w1} and 
\eqref{eq:repeated-conv}, we have
\begin{align*}
w_n^{*(k+1)} 
&= \sum_{r=1}^n {n \choose r} w_r^{*k} w_{n-r} \\
&= \sum_{r=2k}^{n-2} \binom{n}{r} 2k \cdot (2k)!
 r^{r-2k-1} {r \choose 2k} \cdot 2(n-r-1)(n-r)^{n-r-2}\\
&=4k\sum_{r=2k}^{n-2} \binom{n}{r} (r-1)\cdots (r-(2k-1))r^{r-1-(2k-1)}(n-r-1)(n-r)^{n-r-2}.
\end{align*}
Now we introduce a class of polynomials which appears in Abel's generalization of the binomial formula \cite[Section 1.5]{Riordan}: 
\begin{align*}
A_n(x,y; p,q) :=\sum_{r=0}^n \binom{n}{r}(x+r)^{r+p}(y+n-r)^{n-r+q}.
\end{align*}
In particular, when $p=q=-1$, it is known \cite[p.23]{Riordan} that 
\begin{equation}
A_n(x,y; -1,-1)= (x^{-1}+y^{-1})(x+y+n)^{n-1}. 
\label{eq:Anxy}
\end{equation}
Multiplying both sides by $xy$ yields 
\begin{align}\label{eqS}
(x+y)Q(x,y) 
&= xy\sum_{r=0}^n \binom{n}{r}(x+r)^{r-1}(y+n-r)^{n-r-1}\nonumber\\
&= x(x+n)^{n-1}+y(y+n)^{n-1}+xyS(x,y),
\end{align}
where $Q(x,y):=(x+y+n)^{n-1}$ and 
\begin{align*}
S(x,y) :=\sum_{r=1}^{n-1} \binom{n}{r}(x+r)^{r-1}(y+n-r)^{n-r-1}.
\end{align*}
By the generalized Leibniz rule, 
for $p \in \N$, we have
\begin{align*}
\partial_x^p(xS(x,y))
&= p \partial_x^{p-1} S(x,y) 
+ x\partial_x^p S(x,y),\\
\partial_x^p ((x+y)Q(x,y)) 
&= p\partial_x^{p-1} Q(x,y)+(x+y) \partial_x^pQ(x,y),
\end{align*}
which gives
\begin{align}
\partial_x^p \partial_y^2 (xyS(x,y)) \Big|_{x=y=0} 
&= 2pS^{(p-1,1)}(0,0), 
\label{eq:Sderivative}
\\
\partial_x^p \partial_y^2 ((x+y)Q(x,y)) \Big|_{x=y=0} 
&= (p+2)Q^{(p-1,2)}(0,0), 
\label{eq:Qderivative}
\end{align}
where $S^{(p,q)}(x,y):=\partial_x^p \partial_y^q S(x,y)$ and
$Q^{(p,q)}(x,y):= \partial_x^p \partial_y^q Q(x,y)$.

For $k=1,2,\dots, \lfloor n/2 \rfloor$, 
differentiating both sides of \eqref{eqS} $2k$ times with respect to $x$ and twice with respect to $y$ 
and using \eqref{eq:Sderivative} and \eqref{eq:Qderivative} with $p=2k$ yield 
\begin{align*}
\partial_x^{2k} \partial_y^2 ({\rm RHS\ of\
 \eqref{eqS}})\Big|_{x=y=0} 
&= \partial_x^{2k} \partial_y^2 (xyS(x,y))\Big|_{x=y=0}\\
&=4kS^{(2k-1,1)}(0,0)=w_n^{*(k+1)},\\
\partial_x^{2k} \partial_y^2 ({\rm LHS\ of\
 \eqref{eqS}})\Big|_{x=y=0} &=(2k+2)Q^{(2k-1,2)}(0,0)\\
&=2(k+1) \cdot (n-1)\cdots (n-(2k+1))n^{n-1-(2k+1)}\\
&=2(k+1)\cdot (2(k+1))! n^{n-2(k+1)-1}\binom{n}{2(k+1)},
\end{align*}
which complete the proof of \eqref{eq:conv}. 
\end{proof}

Now we derive the leading asymptotic behavior of $N_{\rm{bi}}(n,1)$ as $n \to \infty$. 
\begin{proof}[Proof of Theorem~\ref{thm:asympofun}]
 By \eqref{coefu} and \eqref{eq:conv}, we have
\begin{align*}
N_{\rm{bi}}(n,1)  
&= \sum_{2 \le k \le n/2} (2k)! n^{n-2k-1} {n \choose 2k}\\
&=n^{n-1} \sum_{2 \le k \le n/2}\dfrac{n !}{(n-2k)! n^{2k}}.
\end{align*}
The last summation is similar to the Ramanujan $Q$-function, so we treat this summation in the same way as in \cite[Section 4]{FS09}. Let $k_0$ be an integer such that $k_0=o(n^{2/3})$ and we split the summation into two parts:
\begin{align*}
\sum_{2 \le k \le n/2}\dfrac{n !}{(n-2k)! n^{2k}}=\sum_{2 \le k \le k_0}\dfrac{n !}{(n-2k)! n^{2k}} + \sum_{k_0 < k \le n/2}\dfrac{n !}{(n-2k)! n^{2k}}.
\end{align*}
For $k=o(n^{2/3})$, by \cite[Theorem 4.4]{FS09} we have
\begin{align*}
\dfrac{n !}{(n-2k)! n^{2k}}=e^{-2k^2/n}\left(1+O\left(\frac{k}{n}\right)+O\left(\frac{k^3}{n^2}\right)\right).
\end{align*}
Because the terms in the summation are decreasing in $k$, and $e^{-2k^2/n}$ are exponentially small for $k>k_0$, the second summation is negligible. Therefore,
\begin{align*}
\sum_{2 \le k \le n/2}\dfrac{n !}{(n-2k)! n^{2k}}&=\sum_{2 \le k \le k_0}e^{-2k^2/n}\left(1+O\left(\frac{k}{n}\right)+O\left(\frac{k^3}{n^2}\right)\right) +o(1)\\
 &=\sum_{2 \le k \le k_0}e^{-2k^2/n}+O(1).
\end{align*}
Again, since $e^{-2k^2/n}$ are exponentially small for $k>k_0$, we can take the summation for $2\le k \le n/2$. Therefore, by Euler-Maclaurin's formula we have 
\begin{align*}
\sum_{2 \le k \le n/2}e^{-2k^2/n}=\sqrt{n}\int_0^\infty e^{-2x^2}dx+O(1)=\sqrt{\dfrac{\pi}{8}}\sqrt{n}+O(1),
\end{align*}
which completes the proof.
\end{proof}

\subsection{Asymptotic behavior of the coefficients of $F_2(x,x)$}
\label{subsec:F2}
We deal with the coefficients of the diagonal $F_2(x,x)$, namely $N_{\rm{bi}}(n,2)$, which is defined by \eqref{eq:defN_{bi}(n,k)} with $k=2$. 
From \eqref{eq:Tintro}, 
we have 
\begin{align*}
 Z(x,y) = T_x + T_y 
= \sum_{r,s=0}^{\infty} \frac{(r+s)r^{s-1}s^{r-1}}{r!s!}
 x^r y^s.  
\end{align*}
In particular, by Lemma \ref{keylem1} we have
\begin{align}\label{eq:diagonalZ}
 Z(x,x) 
&= \sum_{r,s=0}^{\infty} 
\frac{(r+s)r^{s-1}s^{r-1}}{r!s!}  x^{r+s}\nonumber\\
&= \sum_{n=1}^{\infty} n 
\left(\sum_{r+s=n} \frac{n! r^{s-1}s^{r-1}}{r!s!}\right)
\frac{x^n}{n!} 
= \sum_{n=1}^{\infty} 2n^{n-1} \frac{x^n}{n!}.
\end{align}
Let $Y(x)$ be the exponential generating function for the number of labeled rooted spanning trees in $K_n$:
\begin{align}\label{eq:egftree}
Y(x):=\sum_{n=1}^{\infty}n^{n-1}\frac{x^n}{n!}.
\end{align}
First we see the formula for the power of $Y(x)$. 
\begin{lem}\label{lem:polyegftree}
For $k=1,2,\dots$, 
\begin{align}\label{eq:polyegftree}
Y(x)^k=\sum_{n=1}^{\infty}k(n-1)(n-2)\cdots(n-(k-1))n^{n-k}\frac{x^n}{n!}.
\end{align}
\end{lem}
\begin{proof}
The proof is by induction in $k$. Assume that \eqref{eq:polyegftree} holds for $k$. Then,
\begin{align*}
Y(x)^{k+1}&=\left(\sum_{n=1}^{\infty}k(n-1)(n-2)\cdots(n-(k-1))n^{n-k}\frac{x^n}{n!}\right)\left(\sum_{n=1}^{\infty}n^{n-1}\frac{x^n}{n!}\right)\\
&=kx^{k+1}\left(\sum_{n=0}^{\infty}(n+k)^{n-1}\frac{x^n}{n!}\right)\left(\sum_{n=0}^{\infty}(n+1)^{n-1}\frac{x^n}{n!}\right)\\
&=kx^{k+1}\sum_{n=0}^{\infty}\left(\sum_{r=0}^n\binom{n}{r}(k+r)^{r-1}(1+n-r)^{n-r-1}\right)\frac{x^n}{n!}.
\end{align*}
Note that by \eqref{eq:Anxy}, 
\begin{align*}
\sum_{r=0}^n\binom{n}{r}(k+r)^{r-1}(1+n-r)^{n-r-1}&=A_{n}(k,1;-1,-1)\\
&=\Bigl(\frac{1}{k}+1\Bigr)(k+1+n)^{n-1},
\end{align*}
so that 
\begin{align*}
Y(x)^{k+1}&=x^{k+1}\sum_{n=0}^{\infty}(k+1)(n+k+1)^{n-1}\frac{x^n}{n!}\\
&=(k+1)\sum_{n=0}^{\infty}(n+1)(n+2)\cdots(n+k)(n+k+1)^{n}\frac{x^{n+k+1}}{(n+k+1)!}\\
&=(k+1)\sum_{n=k+1}^{\infty}(n-1)(n-2)\cdots(n-k)n^{n-(k+1)}\frac{x^n}{n!}.
\end{align*}
Hence, \eqref{eq:polyegftree} holds for $k+1$, and by
 induction this completes the proof.
\end{proof}

Lemma \ref{lem:polyegftree} gives for $a_k \in \R, k=1,2,\dots,$
\begin{align}\label{seriesofY}
\sum_{k=1}^{\infty}a_kY(x)^k=\sum_{n=1}^{\infty}n^{n-1}\left(\sum_{k=1}^{\infty}a_kk\frac{(n-1)(n-2)\cdots(n-(k-1))}{n^{k-1}}\right)\frac{x^n}{n!},
\end{align}
where the summation with respect to $k$ is finite.

From Corollary \ref{cor:wn}, \eqref{eq:diagonalZ} and \eqref{eq:polyegftree},
\begin{eqnarray}
\begin{aligned}\label{eq:diagonalZW}
&Z(x,x)=\sum_{n=1}^{\infty}2n^{n-1}\frac{x^n}{n!}=2Y(x),\\
&W(x,x)=\sum_{n=1}^{\infty}2(n-1)n^{n-2}\frac{x^n}{n!}=Y(x)^2.
\end{aligned}
\end{eqnarray}
Hence, we can express $F_2(x,x)$ by using only $Y(x)$, instead of $Z(x,x)$ and $W(x,x)$. Substituting \eqref{eq:diagonalZW} in \eqref{eq:v} with the notation $Y=Y(x)$, we have 
\begin{align}\label{eq:ExpandF2}
F_2(x,x)&=f_2(2Y,Y^2)
=\frac{Y^5(2+4Y-Y^2)}{12(1-Y)^3(1+Y)^2}\nonumber\\
&=\frac{Y^2-3Y-3}{12} 
- \frac{11}{64(1+Y)} + \frac{1}{32(1+Y)^2}
\nonumber\\ 
&\ \ \ + \frac{143}{192(1-Y)} - \frac{11}{24(1-Y)^2}+\frac{5}{48(1-Y)^3}.
\end{align}
In the case of $K_n$, a similar
expression can be found in \cite[(17)]{W77}. As we will see
below, the last term of \eqref{eq:ExpandF2} determines the
asymptotic behavior of $N_{\rm{bi}}(n,2)$ in
Theorem~\ref{thm:asympofF_2}. 

To obtain the asymptotic behavior of $N_{\rm{bi}}(n,2)$,
from \eqref{eq:ExpandF2}, we only need to estimate coefficients of $\frac{1}{(1-Y)^p}$, $\frac{1}{(1+Y)^p},\ p \in \N$. For fixed $p\in \N$, the tree polynomials $\{t_n(p)\}_{n\ge0}$ are defined by 
\begin{align}\label{eq:deftnp}
\frac{1}{(1-Y(x))^p}=\sum_{n=0}^{\infty}t_n(p)\frac{x^n}{n!}.
\end{align}
This polynomial and their asymptotic behavior of $t_n(p)$ are well studied in \cite{KP89}. 
\begin{lem}[\cite{KP89}]
For fixed $p \in \N$, as $n \to \infty$, 
\begin{align}\label{asympoftnp}
t_n(p)=\frac{\sqrt{2\pi}n^{n-1}}{2^{p/2}}\left(\frac{n^{(p+1)/2}}{\Gamma(p/2)}+\frac{\sqrt{2}p}{3}\frac{n^{p/2}}{\Gamma((p-1)/2)}+O(n^{(p-1)/2})+O(1)\right).
\end{align}
\end{lem}

Hence, we have already obtained the asymptotic behavior of
$\frac{1}{(1-Y)^p}, p \in \N$. 
For $\frac{1}{(1+Y)^p}, p \in \N$, we only give a rough
estimate for coefficients of $\frac{1}{(1+Y)^p}$. By the
binomial expansion and \eqref{seriesofY}, we have 
\begin{align*}
\frac{1}{(1+Y(x))^p}&=\sum_{k=0}^{\infty}\binom{p+k-1}{k}(-1)^kY(x)^k\\
&=1+\sum_{n=1}^{\infty}\left(\frac{n^{n-1}}{\Gamma(p)}\sum_{k=0}^{\infty}\binom{n-1}{k}(-1)^{k+1}\frac{\Gamma(p+k+1)}{n^k}\right)\frac{x^n}{n!},
\end{align*}
so that as $n \to \infty$, 
\begin{align}\label{eq:orderof(1+Y)^p}
\bra x^n \ket \frac{1}{(1+Y(x))^p}&=\frac{n^{n-1}}{\Gamma(p)}\sum_{k=0}^{\infty}\binom{n-1}{k}(-1)^{k+1}\frac{\Gamma(p+k+1)}{n^k}\nonumber\\
&\le t_n(p)=O(n^{n+(p-1)/2}).
\end{align}

Now we are in a position to prove Theorem~\ref{thm:asympofF_2}.  
\begin{proof}[Proof of Theorem~\ref{thm:asympofF_2}]
By \eqref{eq:diagonalZW}, \eqref{eq:ExpandF2}, \eqref{asympoftnp} and
 \eqref{eq:orderof(1+Y)^p}, we obtain the leading asymptotic
 behavior of $N_{\rm{bi}}(n,2)$ as
\begin{align*}
N_{\rm{bi}}(n,2)= \bra x^n \ket F_2(x,x)
&=-\frac{1}{12}n^{n-1}+O(n^n)+O(n^{n+1/2})+ \frac{143}{192}(n^n+O(n^{n-1/2}))\\
&\ \ \  - \frac{11}{24}\Bigl(\sqrt{\frac{\pi}{2}}n^{n+1/2}+O(n^{n})\Bigr)+\frac{5}{48}(n^{n+1}+O(n^{n+1/2}))\\
&=\frac{5}{48}n^{n+1}+O(n^{n+1/2}),
\end{align*}
which completes the proof.
\end{proof}

\section{Another expression for $F_k(x,y)$}

In this section, we introduce the notion of basic graphs obtained from connected bipartite graphs, 
and we give proofs of Theorem~\ref{thm:expressionofF_k} and Theorem~\ref{thm:intro}. 
In a similar way to the proof of Theorem~\ref{thm:expressionofF_k}, 
we give another proof of Proposition~\ref{prop:Uintro}.

\subsection{Proof of Theorem~\ref{thm:expressionofF_k}}
Our proof is based on the combinatorial argument developed in \cite[Section 6]{W77}. 
Firstly, we explain how to obtain a basic graph from a connected bipartite graph. 

Fix $k \ge 2$ and take a labeled connected bipartite
 $(r,s,r+s-1+k)$-graph $G$ whose vertex set is $V=(V_1,V_2)$
 with $|V_1|=r$ and $|V_2|=s$. We delete a leaf and its adjacent edge from $G$, 
and repeat this procedure until vanishing all leaves in the resultant graph. 
Since we delete only one vertex and one edge in each
 procedure, we obtain a labeled connected bipartite $(t,u,t+u-1+k)$-graph without
 leaf for some $t \le r$ and $u \le s$. 
Clearly, the resultant graph does not depend on the
 order of eliminations of leaves, and it is denoted by $G'$. Let
 $V'=(V_1',V_2')$ be the vertex set of the graph $G'$. For each
 vertex $v \in V'$, we call it a \textit{special point}
 if $\deg(v)\ge 3$ and a \textit{normal point} if
 $\deg(v)=2$. 
Let $r_{{\rm sp}}$ and $s_{{\rm sp}}$ be the
 number of special points in $V_1'$ and $V_2'$,
 respectively. 
By applying the handshaking lemma to the graph $G'$, 
we see that $\sum_{v \in V'} ({\rm deg}(v) - 2) = 2(k-1)$ 
and hence 
\begin{align}\label{eq:numofsppoints}
r_{{\rm sp}}+s_{{\rm sp}}\le 2(k-1).
\end{align}
In the graph $G'$, a path whose end vertices are 
 distinct special points is said to be a \textit{special path} 
and a cycle which contains exactly one special point is said to be 
a \textit{special cycle}. 
Since $G'$ is connected and $\deg(v) \ge 2$, it is clear
 that it consists of such
 special paths and cycles which are disjoint except at special points. 
We classify these special paths and cycles into seven types 
 and contract them to the minimal ones as in
 Figure~\ref{fig:specialpaths} to obtain the \textit{basic graph} $\cB(G)$. 

\begin{itemize}
 \item 
An $\alpha_i$\textit{-cycle} is a special cycle with exactly one special
point in $V_i'$ ($i=1,2$). 
By the structure of bipartite graphs, these special cycles 
contain at least three normal points. The minimal $\alpha_i$-cycle has three normal points 
as in Figure~\ref{fig:specialpaths}. 
\item A $\beta_j$\textit{-path} is a special path whose end vertices 
are two distinct special points in $V_j'$ ($j=1,2$). 
By the structure of bipartite graphs, these special paths 
contain at least one normal point. The minimal $\beta_j$-path has only one normal point as in Figure~\ref{fig:specialpaths}. 
\item A special path whose end vertices are special points 
in $V_1'$ and $V_2'$ is called in several ways according to the situation. 
For each pair of special points $v_1 \in V_1'$ and $v_2 \in V_2'$, we have two cases. 
\begin{itemize}
 \item Case(i) there is only one special path connecting $v_1$ and $v_2$: 
such a special path is called a $\gamma$\textit{-path}.  
The length of the minimal $\gamma$-path is one. 
\item Case(ii) there is more than one special path connecting $v_1$ and $v_2$: 
since we are considering a simple graph, there is at most one such a special path of length one, 
i.e., joined by an edge. A special path is called 
a $\delta$\textit{-path} if the length is three or more 
and a $\delta$\textit{-edge} if the length is one. 
The length of the minimal $\delta$-path is three. 
\end{itemize}
\end{itemize}

We decomposed $G'$ into the union of a collection
 of $\alpha_i$-cycles, $\beta_j$-paths, $\gamma$-paths, $\delta$-paths,
 and $\delta$-edges. The \textit{basic graph} $\cB(G)$ is obtained from $G'$ by contracting 
$\alpha_i$-cycles, $\beta_j$-paths, $\gamma$-paths, and $\delta$-paths to
 the minimal ones as in Figure~\ref{fig:specialpaths}. 
In the procedure of contraction, we forget about
 labels of vertices. 
We summarize the contraction procedures below. 
\begin{itemize}
 \item If each $\alpha_i$-cycle ($i=1,2$) contains five or more normal points, we
contract it to the minimal $\alpha_i$-cycle, which has three normal points. 
 \item If each $\beta_j$-path ($j=1,2$) contains three or more normal points, we
contract it to the minimal $\beta_j$-path, which has only one normal point. 
 \item If each $\gamma$-path contains normal points, we
       contract it to the minimal $\gamma$-path, which has no normal points. 
 \item If each $\delta$-path contains four or more normal points, we
contract it to the minimal $\delta$-path, which has two normal points. 
\end{itemize}

\begin{figure}[htbp]
\begin{center}
\includegraphics[width=0.8\hsize]{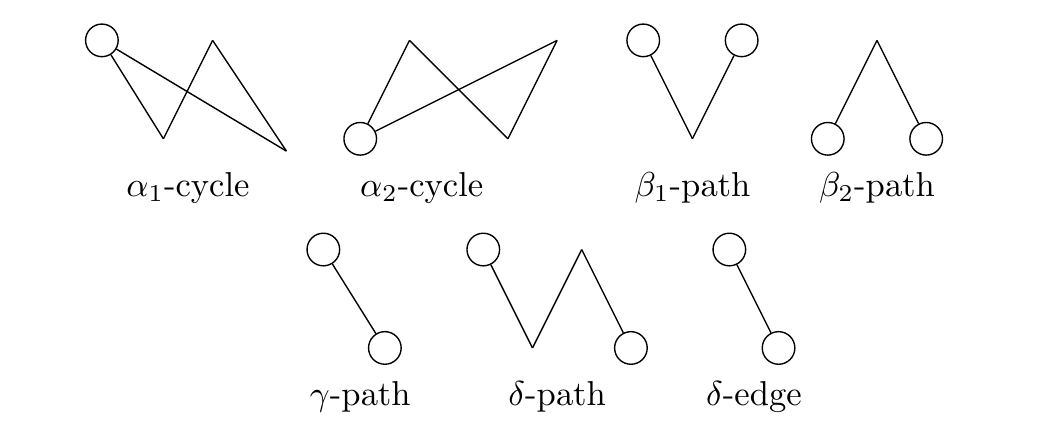}
\end{center}
\caption{Seven types of minimal special paths, cycles and edge. The circles denote special points. }
\label{fig:specialpaths}
\end{figure}

We have seen how to make the basic graph $\cB(G)$
from a given labeled connected bipartite $(r,s,r+s-1+k)$-graph $G$.  
Note that the number of cycles in graphs is invariant by the contractions, so that $\cB(G)$
has just $k$ cycles. 
We will reconstruct labeled bipartite graphs from each basic graph $\cB$ 
and introduce $J_{\cB}(x,y)$ to express $F_k(x,y)$ as sum of $J_{\cB}(x,y)$'s. 

\begin{proof}[Proof of Theorem~\ref{thm:expressionofF_k}]
For a given labeled connected bipartite $(r,s,r+s-1+k)$-graph $G$, 
let $V''=(V_1'',V_2'')$ be the vertex set of $\cB(G)$, 
and also let 
$a_i, b_j, c, d$ and $e$ be the number of
 $\alpha_i$-cycles, $\beta_j$-paths, $\gamma$-paths, $\delta$-paths, and $\delta$-edges in $\cB(G)$, respectively. 
Then, for the number of vertices in $\cB(G)$, we have
\begin{align}\label{eq:numofV_1''}
|V_1''|&=r_{\rm sp}+a_1+2a_2+b_2+d \le t \le r,\\ 
\label{eq:numofV_2''}
|V_2''|&=s_{\rm sp}+2a_1+a_2+b_1+d \le u \le s.
\end{align}
For the number of edges in $\cB(G)$, 
since the same number of vertices and edges are deleted by contraction, we have
\begin{equation}
4a_1+4a_2+2b_1+2b_2+c+3d+e=|V_1''|+|V_2''|+k-1. 
\label{eq:numofE''}
\end{equation}
Combining \eqref{eq:numofV_1''}-\eqref{eq:numofE''} 
and the inequality \eqref{eq:numofsppoints}, we have
\begin{align}\label{eq:a_1tod}
a_1+a_2+b_1+b_2+c+d+e&=r_{\rm sp}+s_{\rm sp}+k-1 \nonumber \\
&\le 3(k-1).
\end{align}
Therefore, if $G$ is a labeled connected bipartite
 $(r,s,r+s-1+k)$-graph, then $\cB(G)$ should satisfy the
 conditions \eqref{eq:numofsppoints}-\eqref{eq:a_1tod}. 
Now we denote the set of all possible basic graphs having
 $k$ cycles by $BG_k$,
 i.e., 
\[
 BG_k := \{\cB(G) : \text{$G$ is a labeled bipartite
 $(r,s,r+s-1+k)$-graph for some $r,s$}. \}
\] 
It follows from \eqref{eq:a_1tod} that $BG_k$ is a
 \textit{finite} set. 

For fixed $\cB \in BG_k$, 
let $j_{\cB}(r,s)$ be the number of labeled connected bipartite
 $(r,s,r+s-1+k)$-graphs $G$ such that $\cB(G) = \cB$. 
We define the exponential generating function of
 $j_{\cB}(r,s)$ as 
\begin{align*}
J_{\cB}=J_{\cB}(x,y):=\sum_{r,s=0}^{\infty}j_{\cB}(r,s)\frac{x^ry^s}{r!s!}.
\end{align*}
We will show below that $J_{\cB}(x,y)$ is expressed by a
 rational function of $T_x$ and $T_y$. 
To this end, we count $j_{\cB}(r,s)$ by reversing the
 procedure of contraction above, i.e., by adding pairs of a normal point and its
 adjacent edge in $\cB$ and rearranging labels of $(r,s)$
 vertices. 
We construct bipartite $(r,s,r+s-1+k)$-graphs from $\cB$ by two steps as
 follows. \\ 
Step 1: 
Take $\cB \in BG_k$. Let $V'' = (V_1'', V_2'')$ be the vertex set of $\cB$ and 
$M:=a_1+a_2+b_1+b_2+c+d$ 
be the number
 of all minimal special paths and cycles in $\cB$ except $\delta$-edges. 
 Take $t$ and $u$ such that $|V_1''| \le t \le r$ and
 $|V_2''| \le u \le s$. 
We label all minimal $\alpha_i$-cycles, $\beta_j$-paths, $\gamma$-paths and 
 $\delta$-paths in $\cB$, say, $\ss_1,\ss_2,\dots,\ss_M$,
 and we add pairs of a normal point and its adjacent edge in these special paths/cycles. 
By the structure of bipartite graphs, for every $j=1,2,\dots, M$, the number of added pairs in
 each $\ss_j$ is even, and the numbers of added normal
 points in $V_1''$ and $V_2''$ are equal, which we denote by $m_j$. 
Hence, a necessary condition for the numbers of added vertices in $V_1''$ and
 $V_2''$ is $t-|V_1''|=u-|V_2''| =\sum_{j=1}^M m_j$. 
Combining \eqref{eq:numofV_1''} and \eqref{eq:numofV_2''} with the
 necessary condition, the non-negative integers $\{m_j\}_{j=1}^M$ satisfy 
\begin{align}
m_1+m_2+\cdots +m_M &= t-(r_{\rm sp}+a_1+2a_2+b_2+d),
\label{eq:addingptsinV_1''}\\
m_1+m_2+\cdots +m_M &= u-(s_{\rm sp}+2a_1+a_2+b_1+d).
\label{eq:addingptsinV_2''}
\end{align}
Let $y_{\cB}(t,u)=y_{\cB}(t,u,r_{{\rm sp}},s_{{\rm sp}},a_1,a_2,b_1,b_2,c,d)$ 
be the number of the solutions $\{m_j\}_{j=1}^M$ of \eqref{eq:addingptsinV_1''} and
 \eqref{eq:addingptsinV_2''}. 
For each solution $\{m_j\}_{j=1}^M$, we obtain 
an unlabeled connected bipartite $(t,u,t+u-1+k)$-graph, 
and hence $y_{\cB}(t,u)$ of those from $\cB$. \\
Step 2: 
Take one of $y_{\cB}(t,u)$ of unlabeled connected bipartite
 $(t,u,t+u-1+k)$-graphs and call its vertices $T_1,\dots, T_t$ and $U_1,\dots, U_u$. 
Let $\cI_{t,u}$ the set of $\{(r_{1i}, s_{1i})\}_{i=1}^t$
 and $\{(r_{2j}, s_{2j})\}_{j=1}^u$ such that 
$r_{1i}\ge1, r_{2j}\ge0, s_{1i}\ge0, s_{2j}\ge 1$,
$\sum\limits_{i=1}^tr_{1i}+\sum\limits_{j=1}^ur_{2j}=r$ and
$\sum\limits_{i=1}^ts_{1i}+\sum\limits_{j=1}^us_{2j}=s$. 
For each $\{(r_{1i}, s_{1i})\}_{i=1}^t$
 and $\{(r_{2j}, s_{2j})\}_{j=1}^u$ in $\cI_{t,u}$, 
we attach a rooted tree of size $(r_{1i}, s_{1i})$ to
 $T_i$ for $i=1,2,\dots,t$ and a rooted tree of size
 $(r_{2j}, s_{2j})$ to $U_j$ for $j=1,2,\dots,u$,
 respectively.  
Let $N(r,s,t,u)$ be the number of these
 bipartite $(r,s,r+s-1+k)$-graphs. Then, by counting $t$
 rooted trees whose roots are in $V_1$ and $u$ rooted trees
 whose roots are in $V_2$, we have 
\begin{align}\label{eq:phi}
N(r,s,t,u)
&=\sum\nolimits'\binom{r}{r_{11},\dots,r_{1t},r_{21},\dots,r_{2u}}\binom{s}{s_{11},\dots,s_{1t},s_{21},\dots,s_{2u}}
 \nonumber \\
&\quad \times 
\prod_{i=1}^tr_{1i}^{s_{1i}}s_{1i}^{r_{1i}-1}\prod_{j=1}^ur_{2j}^{s_{2j}-1}s_{2j}^{r_{2j}},
\end{align}
where the summation $\sum\nolimits'$ is taken over the set $\cI_{t,u}$. \\
By the above two steps, we obtain all labeled connected bipartite $(r,s,r+s-1+k)$-graphs from
the basic graph $\cB$. However, not all of them are different
because of forgetting labels $\ss_1, \dots, \ss_M$ after attaching labeled rooted trees to all vertices. 
Indeed, if $g_{\cB}$ is the number of automorphisms of
 $\cB$, then every graph appears exactly $g_{\cB}$ times. 
Hence, we have
\begin{align*}
j_{\cB}(r,s)=\sum_{\substack{|V_1''| \le t \le r \\ |V_2''| \le u \le s}}\frac{y_{\cB}(t,u)N(r,s,t,u)}{g_{\cB}}.
\end{align*}
Using this, we have
\begin{align}\label{eq:J_B(x,y)}
J_{\cB}(x,y)=\frac{1}{g_{\cB}}\sum_{\substack{|V_1''|\le t\\ |V_2''|\le u}}y_{\cB}(t,u)\sum_{\substack{t\le r \\ u \le s}}N(r,s,t,u)\frac{x^ry^s}{r!s!}.
\end{align}
For the summation in $r$ and $s$, by \eqref{eq:phi}, we have
\begin{align*}
\sum_{\substack{t\le r\\ u \le s}}N(r,s,t,u)\frac{x^ry^s}{r!s!}
&=\sum_{\substack{t\le r\\ u \le s}}\sum\nolimits'\prod_{i=1}^tr_{1i}^{s_{1i}}s_{1i}^{r_{1i}-1}\frac{x^{r_{1i}}y^{s_{1i}}}{r_{1i}!s_{1i}!}\prod_{j=1}^ur_{2j}^{s_{2j}-1}s_{2j}^{r_{2j}}\frac{x^{r_{2j}}y^{s_{2j}}}{r_{2j}!s_{2j}!}\nonumber\\
&=T_x^tT_y^u.
\end{align*}
On the other hand, by a straightforward calculation, we have
\begin{align}\label{eq:sumofy_{B}(t,u)T_x^tT_y^u}
\sum_{\substack{|V_1''|\le t\\ |V_2''|\le
 u}}y_{\cB}(t,u)T_x^tT_y^u 
&=\sum_{\substack{|V_1''|\le t\\ |V_2''|\le u}}\ \sum_{\substack{m_1,\dots,m_M\\
\sum m_j=t-|V_1''|=u-|V_2''|}}T_x^tT_y^u\nonumber\\
&=T_x^{|V_1''|}T_y^{|V_2''|}\sum_{\substack{|V_1''|\le t\\
 |V_2''|\le u}}\ \sum_{\substack{m_1,\dots, m_M\nonumber\\
\sum m_j=t-|V_1''|=u-|V_2''|}}(T_xT_y)^{\sum_{j=1}^{M} m_j}\\
&=T_x^{|V_1''|}T_y^{|V_2''|}\sum_{n\ge0}\ \sum_{\substack{m_1,\dots,m_M\\
\sum m_j=n}}(T_xT_y)^{\sum_{j=1}^{M} m_j}\nonumber\\
&=T_x^{|V_1''|}T_y^{|V_2''|}\prod_{j=1}^{M}\Bigl(\sum_{m_j\ge0}(T_xT_y)^{m_j}\Bigr)\nonumber\\
&=T_x^{|V_1''|}T_y^{|V_2''|}(1-T_xT_y)^{-M}.
\end{align}
Combining \eqref{eq:numofV_1''}, \eqref{eq:numofV_2''}, \eqref{eq:a_1tod}, \eqref{eq:J_B(x,y)} and \eqref{eq:sumofy_{B}(t,u)T_x^tT_y^u}, we obtain \eqref{eq:J_B}. Since non-isomorphic basic graphs with $k$ cycles lead non-isomorphic labeled connected bipartite $(r,s,r+s-1+k)$-graphs, taking a summation $J_{\cB}$ with respect to $\cB \in BG_k$, we obtain \eqref{eq:expressionofF_k}, which completes the proof.
\end{proof}
We give an example of Theorem~\ref{thm:expressionofF_k} for $k=2$. 
\begin{ex}[$k=2$]
Let us consider all the basic graphs for $k=2$ and compute $F_2(x,y)$.
From the conditions \eqref{eq:numofsppoints} and \eqref{eq:a_1tod}, we have 
\begin{align*} 
& r_{\mathrm{sp}} + s_{\mathrm{sp}} \leq 2, \\
& a_{1}+a_{2}+b_{1}+b_{2}+c+d+e =r_{\mathrm{sp}}+s_{\mathrm{sp}}+1. 
\end{align*}
As a result, the possible combinations of 
numbers of special points are 
$(r_{\mathrm{sp}},s_{\mathrm{sp}})=(1,0),(0,1),(1,1),(2,0),(0,2)$.
We compute $J_{\mathcal{B}}$ for each of these cases.
For instance, the calculation procedure is described below for the case of $(r_{\mathrm{sp}},s_{\mathrm{sp}})=(1,0)$.
First, consider the numbers 
of cycles, paths and edges that make up the basic graphs. 
The following should be obvious.
By using 
\begin{align*} 
a_{1}+a_{2}+b_{1}+b_{2}+c+d+e=2,
\end{align*}
we have $(a_1,a_2,b_1,b_2,c,d,e)=(2,0,0,0,0,0,0)$.
As a result, the basic graph is a combination of two $\alpha_1$-cycles.
This is the upper left graph in Figure~\ref{fig:basic_graphs}.
We define this basic graph as $\cB_1$.
Note that basic graphs are unlabeled.

Next, let us compute the number of graph automorphism $g_{\cB_1}$.
We label each of the vertices appropriately.
For the labeled basic graph, there are $2!$ ways to arrange the two $\alpha_1$-cycles.
There are two possible ways to label the vertices of each $\alpha_1$-cycle: 
$1 \to 2 \to 3 \to 4 \to 1$ with the special point as 1, or in reverse $1 \to 4 \to 3 \to 2 \to 1$. 
Therefore $g_{\cB_1}=2!\times 2^2 = 8$.
Consequently, from \eqref{eq:J_B} we obtain 
\begin{align*} 
J_{\cB_1}(x,y)=\frac{T_x^3T_y^4}{8(1 - T_xT_y)^2}.
\end{align*}
We can derive the others by the same calculation.
Therefore,
\begin{align*}
\sum_{\mathcal{B} \in B G_{2}} J_{\mathcal{B}}(x, y) =&\frac{T_x^3T_y^4}{8(1 - T_xT_y)^2} + \frac{T_x^4T_y^3}{8(1 - T_xT_y)^2} + \frac{T_x^4T_y^5}{8(1-T_x T_y)^3} \\
&   +\frac{T_x^2T_y^3}{12(1-T_x T_y)^3} + \frac{T_x^5T_y^4}{8(1-T_x T_y)^3} + \frac{T_x^3T_y^2}{12(1-T_x T_y)^3}\\
            &+\frac{T_x^4T_y^4}{6(1-T_x T_y)^3}+\frac{T_x^3T_y^3}{2(1-T_x T_y)^2} + \frac{T_x^4T_y^4}{4(1-T_x T_y)^3}\\
           =&\frac{W^{2}(2+3 W)}{24(1-W)^{3}} Z+\frac{W^{3}(6-W)}{12(1-W)^{3}} \\
           =&F_{2}(x, y),
\end{align*}
where the nine terms correspond to the nine basic graphs in Figure~\ref{fig:basic_graphs}, respectively. 
Hence, the result of the calculation by using basic graphs is consistent with $F_{2}(x,y)=f_2(Z,W)$. 

\begin{figure}[htbp]
\begin{center}
\includegraphics[width=0.8\hsize]{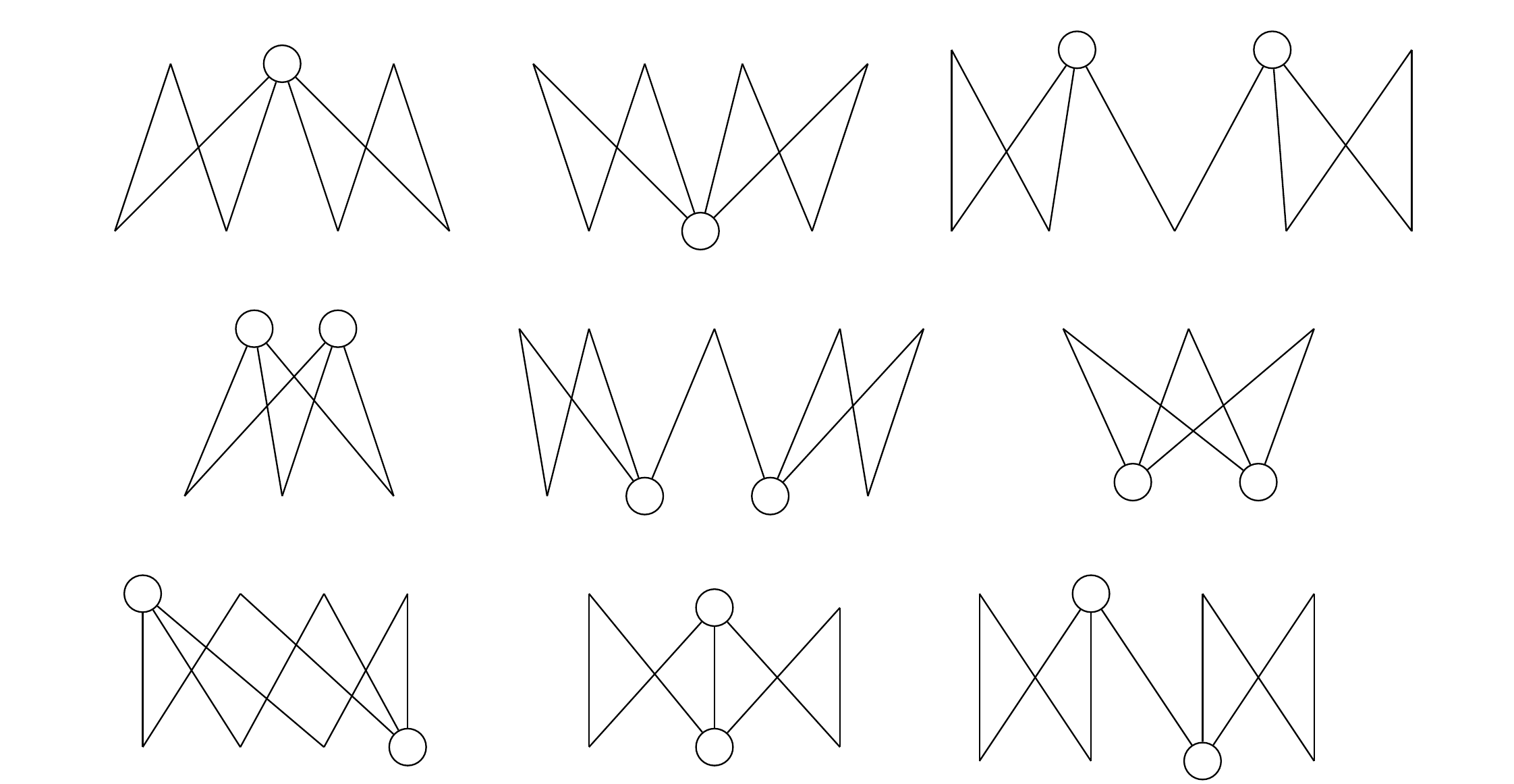}
\end{center}
\caption{Basic graphs for $k=2$}
\label{fig:basic_graphs}
\end{figure}
\end{ex}

\subsection{Proof of Theorem~\ref{thm:intro}}

From \eqref{eq:expressionofF_k}, \eqref{eq:J_B} and
 \eqref{eq:a_1tod}, we see that 
\begin{equation}
 F_k(x,y) 
= \frac{1}{(1-T_xT_y)^{3(k-1)}} 
\sum_{\cB \in BG_k} \frac{1}{g_{\cB}} T_x^{r_{\cB}}
T_y^{s_{\cB}} (1-T_xT_y)^{p_{\cB}}, 
\label{eq:poly}
\end{equation}
where $r_{\cB} = r_{{\rm sp}}+a_1+2a_2+b_2+d$, 
$s_{\cB} = s_{{\rm sp}}+2a_1+a_2+b_1+d$, and 
\begin{align}
p_{\cB} &= 3(k-1) - (a_1+a_2+b_1+b_2+c+d) \nonumber\\
&= 2(k-1) - (r_{{\rm sp}} + s_{{\rm sp}}) + e \nonumber\\
&= \sum_{v \in \text{special points}} (\deg(v) -3) + e\nonumber\\
&\ge 0.
\label{eq:p_B} 
\end{align}
Note that there are some basic graphs $\cB \in BG_k$ such that $p_{\cB}=0$. 
For example,
we can construct a basic graph $\tilde{\cB}$ with $r_{{\rm sp}}=2(k-1)$, $a_1=2$, $b_1=3k-5$ and other constants vanishing as follows: 
we label all $2(k-1)$ special points in $V_1$, say, $r_1, r_2, \dots, r_{2(k-1)}$. We attach an $\alpha_1$-cycle to each of $r_1$ and $r_{2(k-1)}$, 
and then connect $r_{2j-1}$ with $r_{2j}$ ($j=1,2,\dots,k-1$) by a $\beta_1$-path 
and $r_{2j}$ with $r_{2j+1}$ ($j=1,2,\dots,k-2$) by two $\beta_1$-paths. 
Then, we obtain $\tilde{\cB}$. 
Remark that in the case of $k=2$, $\tilde{\cB}$ corresponds to the top-right graph in Figure~\ref{fig:basic_graphs}. 
Clearly, $s_{{\rm sp}}=e=0$ holds for all $k\ge2$, and the calculation in \eqref{eq:p_B} gives $p_{\tilde{\cB}}=0$.
From this observation, the numerator of the right-hand side of \eqref{eq:poly} turns out to be a polynomial of the following form 
\[
Q(x,y) = \sum_{i=1}^m C_i x^{a_i} y^{b_i} +\sum_{i=m+1}^{m+n} C_i x^{a_i} y^{b_i} (1- xy)^{p_i},
\]
for some positive integers $m$ and $n$. Here $a_i, b_i$ are non-negative integers, $p_i$ is a positive integer 
and $C_i > 0$ for all $i$.
If $Q(x,y)$ has a factor $1-xy$, 
plugging $y=x^{-1}$ in both sides yields $0=\sum_{i=1}^m C_i > 0$, which is a contradiction. Hence,  
 $Q(x,y)$ does
 not have the factor $1-xy$, which implies that 
the numerator of the right-hand side of \eqref{eq:poly} does not have a factor $1-T_xT_y$. 

We will prove  \eqref{eq:generalfk}. 
Assume $\cB=(V,E)$ with bipartition $V=V_1 \sqcup V_2$. 
We have the following lemma.
\begin{lem}\label{lemma:difdeg}
For any $\cB \in BG_k$, $\big| |V_1|-|V_2| \big| \le k-1$.
\end{lem}
\begin{proof}
Since $\deg(v) \ge 2$ for every vertex $v$ in a basic graph $\cB$, we see that 
\[
 |E| - |V| - \big| |V_1|-|V_2| \big|
= |E| - 2 \max(|V_1|, |V_2|) \ge 0.  
\]
On the other hand, $|E|-|V|=k-1$ since $\cB$ is connected and $k$ is the Betti number. Therefore, 
$\big| |V_1|-|V_2| \big| \le k-1$. 
\end{proof}
For $\cB=(V_1 \sqcup V_2,E) \in BG_k$, 
there exists a unique basic graph $\cB'=(V_1' \sqcup V_2',E') \in BG_k$ such that 
$V_1'=V_2, V_2'=V_1 $, and $E'=E$. 
Let $\pi: BG_k \to BG_k$ be a mapping defined by $\pi(\cB)=\cB'$.
Then $\pi $ is an involution, and we have
\begin{align}
g_{\pi(\cB)}=g_{\cB},
\quad r_{\pi(\cB)}=s_{\cB},
\quad s_{\pi(\cB)}=r_{\cB},
\quad p_{\pi(\cB)}=p_{\cB},
\label{eq:invo}
\end{align}
where $r_\cB=|V_1|$, $s_\cB=|V_2|$.
From this involution with \eqref{eq:invo}, 
the numerator of \eqref{eq:poly} turns out to be
\begin{align}
&\sum_{\cB \in BG_k} \frac{1}{g_{\cB}} T_x^{r_{\cB}}
T_y^{s_{\cB}} (1-T_xT_y)^{p_{\cB}}\nonumber\\
&=\sum_{\substack{\cB \in BG_k \\ r_{\cB}>s_{\cB}}}\frac{(T_xT_y)^{s_{\cB}}}{g_{\cB}} 
(T_x^{q_{\cB}}+T_y^{q_{\cB}})
 (1-T_xT_y)^{p_{\cB}}
+\sum_{\substack{\cB \in BG_k \\ r_{\cB}=s_{\cB}}}\frac{(T_xT_y)^{r_{\cB}}}{g_{\cB}} (1-T_xT_y)^{p_{\cB}}, 
\label{eq:poly2}
\end{align}
where $q_{\cB} = r_{\cB} - s_{\cB}$. Since $T_x^{q_{\cB}}+T_y^{q_{\cB}}$ is a polynomial of $Z$ of degree $q_{\cB}$ with coefficients being polynomials of $W$, 
so is the right-hand side of \eqref{eq:poly2} but the degree is equal to 
$\max_{\cB \in BG_k, r_{\cB} > s_{\cB}} q_{\cB}$. 
Now we consider a basic graph $\hat{\cB}$ which has one special point in $V_2$ and $k$ $\alpha_2$-cycles. Clearly, $r_{\hat{\cB}}=2k, s_{\hat{\cB}}=k+1$ hold, and hence $q_{\hat{\cB}}  =k-1$. This together with Lemma~\ref{lemma:difdeg} implies $\max_{\cB \in BG_k, r_{\cB} > s_{\cB}} q_{\cB}  = k-1$. 
Since $\cB$ is a simple graph and has at least one cycle, 
we have $r_{\cB}, s_{\cB}\ge 2$.
Then, the right hand side of \eqref{eq:poly2} has a factor $(T_xT_y)^2=W^2$. 
Thus the proof of \eqref{eq:generalfk} is completed. 

\subsection{Combinatorial proof of Proposition \ref{prop:Uintro}}

Finally, we remark on another proof of Proposition 
\ref{prop:Uintro} using a similar argument in the proof of 
Theorem~\ref{thm:expressionofF_k}, which
is a bipartite version of the combinatorial argument discussed in
\cite[Section 5]{W77}. 
We use the same notation as above. 
In the preliminary step, we delete leaves and adjacent edges repeatedly. 
In this case, by this procedure, 
we obtain the unique cycle of length, say $2t$. 
Let $r,s \ge 2$ be fixed and $V''
= (V_1'', V_2'')$ be a vertex set. Take $t$ such that $2 \le
t \le \min\{r,s\}$, and consider an unlabeled bipartite
unicyclic graph whose length of the cycle is $2t$. Clearly,
$|V_1''|=|V_2''|=t$. For each of vertices of this graph, we
attach a rooted tree in a similar way of Step 2 in the proof
of Theorem \ref{thm:expressionofF_k}. To create $2t$ rooted
trees, we partition $(r,s)$ vertices into $2t$ vertex sets,
and all of these partitions are in $\cI_{t,t}$. By this
procedure, we obtain $N(r,s,t,t)$ labeled connected
bipartite $(r,s,r+s)$-graphs, where $N(r,s,t,u)$ is
defined in \eqref{eq:phi}. For each of the obtained graphs,
there are $2t$ automorphisms due to the cycle and labels of
roots of rooted trees. Let $j(r,s)$ be the number of labeled
connected bipartite $(r,s,r+s)$-graphs. Then, we have 
\begin{align*}
j(r,s)=\sum_{2 \le t \le \min\{r,s\}}\frac{N(r,s,t,t)}{2t}.
\end{align*}
Let $J(x,y)$ be the exponential generating function for $j(r,s)$, and we have
\begin{align*}
J(x,y)&=\sum_{r,s=0}^{\infty}j(r,s)\frac{x^ry^s}{r!s!}
=\sum_{t=2}^{\infty}\frac{1}{2t}\sum_{\substack{t \le r \\ t\le s}}\frac{N(r,s,t,t)}{r!s!}x^ry^s\\
&=\frac{1}{2}\sum_{t=2}^{\infty}\frac{(T_xT_y)^t}{t}
=-\frac{1}{2}(\log(1-T_xT_y)+T_xT_y)=F_1(x,y),
\end{align*}
which completes the combinatorial proof of Proposition \ref{prop:Uintro}.

\section{Proof of Theorem~\ref{thm:generalfk}}\label{sec:asymp_equality}

In this section, we prove the asymptotic equality \eqref{eq:f(n,n+k)asmthm} for $k \ge 2$. 
In Subsection \ref{subsec:6.1}, for each basic graph $\cB
 \in BG_k$, we find the leading term of $J_{\cB}(x,x)$ by a combinatorial argument, where the multigraph $\cB^*$ obtained from $\cB$ by contraction plays an important role. 
In Subsection \ref{subsec:6.2}, we introduce the basic
 graphs on complete graphs as discussed in \cite{W77} and
 give a similar discussion in Subsection~\ref{subsec:6.1}, and in Subsection
 \ref{subsec:6.3}, we derive the leading asymptotic behavior of
$N_{\rm{bi}}(n,k)$ defined by \eqref{eq:defN_{bi}(n,k)}. 
Through the existence of multigraphs, 
we will see the correspondence between basic graphs on complete bipartite graphs and those on complete graphs. 

\subsection{Basic graphs $\cB$ and $J_{\cB}(x,x)$}\label{subsec:6.1}
Let us recall again $Y(x)$ in \eqref{eq:egftree} representing exponential generating function for labeled rooted trees.
From \eqref{eq:diagonalZW}, $Z(x,x)=2Y(x)$ and $W(x,x)=Y(x)^2$. Recall that
$$
Z(x,y)=T_x(x,y)+T_y(x,y),\ \ W(x,y)=T_x(x,y)T_y(x,y).
$$
Solving these equations, we have
\begin{align}
T_x(x,x)=T_y(x,x)=Y(x).
\label{eq:T_xT_yY}
\end{align}
From Theorem~\ref{thm:expressionofF_k}, for $k\ge2$ we have
\begin{align*}
F_k(x,x)=\sum_{\cB \in BG_k} J_{\cB}(x,x),
\end{align*}
where
\begin{equation}
J_{\cB}(x,x)=\dfrac{Y^{L}}{g_{\cB}(1-Y^2)^{M}}, 
\label{eq:J_Bdiag}
\end{equation}
with $M=M(\cB):=a_1+a_2+b_1+b_2+c+d$ and $L=L(\cB):=M+r_{{\rm sp}}+s_{{\rm sp}}+2a_1+2a_2-c+d$. 
These constants are determined by $\cB$. In this section, we also use the notation $a_1=a_1(\cB), r_{{\rm sp}}=r_{{\rm sp}}(\cB)$, and so on.
We easily see the following. 
\begin{lem} \label{lemma:decomJB}
For $\cB \in BG_k$, there exist unique constants 
$\{\fa_i(\cB)\}_{i=1}^M, \{\fb_i(\cB)\}_{i=1}^M, \{\fc_j(\cB)\}_{j=0}^{L-2M}$ 
such that
\begin{equation}
J_{\cB}(x,x)
=\sum_{i=1}^M \left(\frac{\fa_i(\cB)}{(1-Y)^i}+\frac{\fb_i(\cB)}{(1+Y)^i}\right)+\sum_{j=0}^{L-2M}\fc_j(\cB)Y^j.
\label{eq:J_B_2}
\end{equation}
In particular, 
\begin{equation}
\fa_M(\cB)=\frac{1}{g_{\cB}2^{M}}, \ \ \fb_M(\cB)=\frac{(-1)^L}{g_{\cB}2^{M}}.
\label{eq:coefJBM}
\end{equation}

\end{lem}
\begin{proof}
Multiplying both sides of \eqref{eq:J_Bdiag} and \eqref{eq:J_B_2} by $(1-Y^2)^M$ and substituting $Y = \pm 1$ 
yield \eqref{eq:coefJBM}. 
\end{proof}

For each $\cB \in BG_k$, we contract its
special cycles and paths and ignore the vertex set $V=(V_1,V_2)$. 
By this procedure, we obtain a multigraph $\cB^*$ from
$\cB$. Let $MG_k$ be the set of all multigraphs obtained
from $BG_k$ by this procedure. Define a mapping $\phi: BG_k \to MG_k$ by $\phi(\cB)=\cB^*$ and for $\cB^* \in MG_k$, 
$$
\phi^{-1}(\cB^*):=\{\cB \in BG_k : \phi(\cB)=\cB^*\}.
$$
All basic graphs which belong to $\phi^{-1}(\cB^*)$ have the
same number of special cycles and the same total number of special paths and edges. 
In what follows in this section, we only consider basic graphs $\cB \in
BG_k$ such that $M(\cB)=3(k-1)$ and multigraphs $\cB^* \in
MG_k$ obtained from such basic graphs $\cB$. From
\eqref{eq:a_1tod}, it follows that such a $\cB$ has no $\delta$-edge, i.e., $e(\cB)=0$ and $r_{{\rm sp}}(\cB)+s_{{\rm sp}}(\cB)=2(k-1)$ holds. 
For given $\cB^* \in MG_k$, we divide the set $\phi^{-1}(\cB^*)$ by pairs of 
$(r_{{\rm sp}}(\cB), s_{{\rm sp}}(\cB))$. 
For $i=0,\dots, 2(k-1)$, define
\begin{align*}
 \phi^{-1}(\cB^*)^{(2(k-1)-i,i)}:=\{\cB \in  \phi^{-1}(\cB^*) : (r_{{\rm sp}}(\cB), s_{{\rm sp}}(\cB))=(2(k-1)-i,i)\}.
\end{align*}
Then, we have
\begin{align}
\phi^{-1}(\cB^*)=\bigsqcup_{0\le i \le 2(k-1)} \phi^{-1}(\cB^*)^{(2(k-1)-i,i)}.
\label{eq:divset}
\end{align}
Note that $\phi^{-1}(\cB^*)^{(2(k-1),0)}$ and
$\phi^{-1}(\cB^*)^{(0,2(k-1))}$ are singletons, and each of
the element is determined by $\cB^*$ in a clear way. Indeed,
if $\cB^*$ has self-loops, replace them 
to minimal $\alpha_1$-cycles. Also, if $\cB^*$ has single
edges or multiple edges, replace them to
$\beta_1$-paths. Putting all vertices of $\cB^*$ in $V_1$
and by this procedure, we obtain a basic graph $\cB \in
\phi^{-1}(\cB^*)^{(2(k-1),0)}$. Clearly, the obtained graph
$\cB$ is unique. In a similar way, we have 
a unique element in
$\phi^{-1}(\cB^*)^{(0,2(k-1))}$. For the following
discussion, we denote by $\cB_{{\rm id}}$ the unique element in $\phi^{-1}(\cB^*)^{(2(k-1),0)}$. 

\begin{lem}\label{lemma:eqauto}
Let $\cB^* \in MG_k$ be given. Then, for $i=0,\dots, 2(k-1)$, 
\begin{align}
\sum_{\substack{\cB \in \phi^{-1}(\cB^*)^{(2(k-1)-i,i)} \\ M(\cB)=3(k-1)}}\frac{1}{g_{\cB}}=\binom{2(k-1)}{i}\frac{1}{g_{\cB_{{\rm id}}}}.
\label{eq:nomauto}
\end{align}
\end{lem}
\begin{proof} 
For $i=0$, $\cB_{{\rm id}}$ and $g_{\cB_{{\rm id}}}$ are determined by $\cB^*$. Clearly, 
$r_{{\rm sp}}(\cB_{{\rm id}})=2(k-1)$ holds. We label all
 $2(k-1)$ special points of $\cB_{{\rm id}}$, and we
 construct $\cB \in \phi^{-1}(\cB^*)^{(2(k-1)-i,i)},
 i=0,\dots, 2(k-1)$ by the following way. For given $i$, we
 choose $\binom{2(k-1)}{i}$ labeled special points in
 $\cB_{{\rm id}}$, and we put these points in $V_2$ without
 changing the connectivity of the vertices. Here, for each $\beta_1$-path in $\cB_{{\rm
 id}}$, delete or add a normal point to create
 $\gamma$-path or $\delta$-path from it. Then, we have
 $\binom{2(k-1)}{i}$ basic graphs
 $\cB_1,\cdots,\cB_{\binom{2(k-1)}{i}}$ from $\cB_{{\rm
 id}}$ which satisfy $(r_{{\rm sp}}, s_{{\rm
 sp}})=(2(k-1)-i,i)$. Since this procedure does not change
 the connectivity of graphs, the numbers of the automorphisms of obtained graphs $\{\cB_{\ell}\}_{\ell=1}^{\binom{2(k-1)}{i}}$ are $g_{\cB_{{\rm id}}}$. To obtain $ \phi^{-1}(\cB^*)^{(2(k-1)-i,i)}$, we forget all labels of special points of $\{\cB_{\ell}\}_{\ell=1}^{\binom{2(k-1)}{i}}$. Nevertheless each of ${\binom{2(k-1)}{i}}$ unlabeled graphs may not be different, we have
\begin{align*}
\sum_{\substack{\cB \in \phi^{-1}(\cB^*)^{(2(k-1)-i,i)} \\ M(\cB)=3(k-1)}}\frac{1}{g_{\cB}}
=\sum_{\ell=0}^{{\binom{2(k-1)}{i}}}\frac{1}{g_{\cB_{\ell}}}
=\binom{2(k-1)}{i}\frac{1}{g_{\cB_{{\rm id}}}}.
\end{align*}
\end{proof}

\begin{prop}\label{prop:coef1/(1+Y)^M}
For $\cB^* \in MG_k$,
\begin{align*}
\sum_{\substack{\cB \in \phi^{-1}(\cB^*) \\ M(\cB)=3(k-1)}} \fb_{3(k-1)}(\cB)=0.
\end{align*}
\end{prop}
To show Proposition \ref{prop:coef1/(1+Y)^M}, 
from \eqref{eq:coefJBM} it is sufficient to prove that for $\cB^* \in MG_k$,
\begin{align}
\sum_{i=0}^{2(k-1)}\sum_{\substack{\cB \in \phi^{-1}(\cB^*)^{(2(k-1)-i,i)} \\ M(\cB)=3(k-1)}} \frac{(-1)^{L(\cB)}}{g_{\cB}}=0.
\label{eq:reduction1}
\end{align}
For the signature of $(-1)^{L(\cB)}$, we have the following lemma.
\begin{lem}\label{lemma:signL}
Suppose that for given $\cB^* \in MG_k$, there exists $\cB \in \phi^{-1}(\cB^*)$ such that $M(\cB)=3(k-1)$. Then, for $i$ and $\cB \in \phi^{-1}(\cB^*)^{(2(k-1)-i,i)}$, 
\begin{align*}
(-1)^{L(\cB)}=(-1)^{k-1+i}.
\end{align*}
\end{lem}
\begin{proof}
Recall that for $\cB \in BG_k$, $L(\cB)=M+r_{{\rm sp}}+s_{{\rm sp}}+2a_1+2a_2-c+d$. By the assumption, we have
\begin{align}
M=3(k-1),\ \ r_{{\rm sp}}+s_{{\rm sp}}=2(k-1),
\label{eq:numofMsp}
\end{align}
which give
\begin{align*}
(-1)^{L(\cB)}=(-1)^{(k-1)-c(\cB)+d(\cB)}.
\end{align*}
By the equation \eqref{eq:numofMsp}, for any considered
 $\cB$, degrees of each special point in $\cB$ are three. In
 particular, so are that in $\cB_{{\rm id}}$. It follows
 that each of special points $v \in \cB_{{\rm id}}$
 satisfies one of the following; $v$
 has an $\alpha_1$-cycle and a $\beta_1$-path, or a
 $\beta_1$-path connected to $v_1$ and two $\beta_1$-paths
 connected to $v_2$, or three $\beta_1$-paths connected to
 $v_3$, where $v_i, i=1,2,3$ are different special points of
 $\cB_{{\rm id}}$. Note that each of $\cB$ is obtained from
 $\cB_{{\rm id}}$ by putting $i$ special points in $V_1$
 into $V_2$ and replacing $\beta_1$-paths with $\gamma$-
 or $\delta$-paths in the same way as in the proof of Lemma \ref{lemma:eqauto}. Hence, for any special point $v' \in V_2$ of $\cB$, the difference of the numbers of $\gamma$- and $\delta$-paths connected to $v'$ is odd. Therefore, we have $(-1)^{(k-1)-c(\cB)+d(\cB)}=(-1)^{(k-1)+i}$, which completes the proof.
\end{proof}

\begin{proof}[Proof of Proposition \ref{prop:coef1/(1+Y)^M}]
By Lemmas \ref{lemma:eqauto} and \ref{lemma:signL}, we have
\begin{align*}
({\rm LHS\ of}\  \eqref{eq:reduction1})
&=\sum_{i=0}^{2(k-1)}(-1)^{k-1+i}\sum_{\substack{\cB \in \phi^{-1}(\cB^*)^{(2(k-1)-i,i)} \\ M(\cB)=3(k-1)}} \frac{1}{g_{\cB}}\\
&=\frac{(-1)^{k-1}}{g_{\cB_{{\rm id}}}}\sum_{i=0}^{2(k-1)}\binom{2(k-1)}{i}(-1)^{i}=0.
\end{align*}
Hence, equation \eqref{eq:reduction1} holds and Proposition \ref{prop:coef1/(1+Y)^M} is proved.
\end{proof}

Proposition \ref{prop:coef1/(1+Y)^M} shows that for any $\cB^* \in MG_k$, $\sum_{\cB \in \phi^{-1}(\cB^*)}J_{\cB}(x,x)$ does not have the terms of  $(1+Y)^{-3(k-1)}$, 
and so the leading asymptotic behavior of $F_k(x,x)$ is determined by the summation of $\frac{\fa_{3(k-1)}(\cB)}{(1-Y)^{3(k-1)}}$. We give the exact value of the coefficient of the summation.

\begin{prop}\label{prop:coef1/(1-Y)^M}
Suppose that for given $\cB^* \in MG_k$, there exists $\cB \in \phi^{-1}(\cB^*)$ such that $M(\cB)=3(k-1)$. Then,
\begin{align*}
\sum_{\substack{\cB \in \phi^{-1}(\cB^*) \\ M(\cB)=3(k-1)}} \fa_{3(k-1)}(\cB)=\frac{1}{2^{k-1}}\frac{1}{g_{\cB_{{\rm id}}}},
\end{align*}
where $\cB_{{\rm id}} \in \phi^{-1}(\cB^*)$ is uniquely determined from $\cB^*$.  
\end{prop}
\begin{proof}
By \eqref{eq:coefJBM} and Lemma~\ref{lemma:eqauto}, 
we have
\begin{align*}
\sum_{\substack{\cB \in \phi^{-1}(\cB^*) \\ M(\cB)=3(k-1)}} \fa_{3(k-1)}(\cB)
&=\frac{1}{2^M}\sum_{i=0}^{2(k-1)}\sum_{\substack{\cB \in \phi^{-1}(\cB^*)^{(2(k-1)-i,i)} \\ M(\cB)=3(k-1)}}\frac{1}{g_{\cB}}\\
&=\frac{1}{2^Mg_{\cB_{{\rm id}}}}\sum_{i=0}^{2(k-1)}\binom{2(k-1)}{i}\\
&=\frac{1}{2^{k-1}}\frac{1}{g_{\cB_{{\rm id}}}},
\end{align*}
where we used $M=3(k-1)$ in the last equation.
\end{proof}

\subsection{Basic graphs on complete graphs}\label{subsec:6.2}
Now we consider the correspondence of $\cB \in BG_k$ to a basic graph with respect to complete graphs $\{K_n\}_{n\ge1}$. A basic graph $\cA$ on $\{K_n\}_{n\ge1}$ consists of the following four types of (minimal) special cycle, paths and edge as in Figure \ref{fig:1}. For details, see \cite[Section 6]{W77}. 
An example for case $k=2$ is shown in Figure~\ref{fig:2}.

\begin{figure}[htbp]
\begin{center}
\includegraphics[width=0.9\hsize]{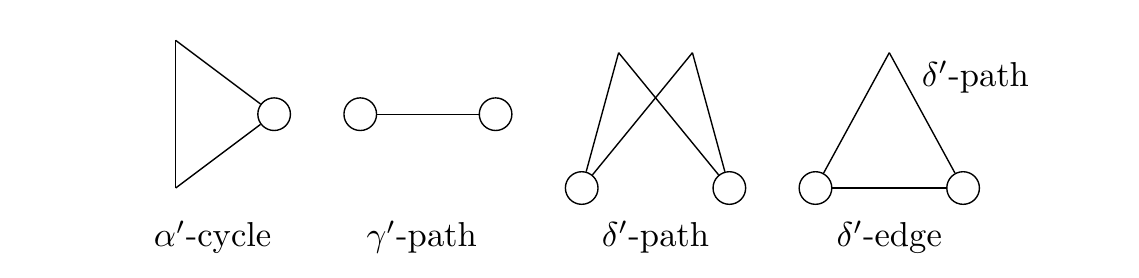}
\end{center}
\caption{
Four types of special cycle and paths of basic graphs on $\{K_n\}_{n\ge1}$. 
The circles denote special points.
}
\label{fig:1}
\end{figure}

\begin{figure}[htbp]
  \begin{minipage}[b]{0.45\linewidth}
    \centering
    \includegraphics[scale=1.0]{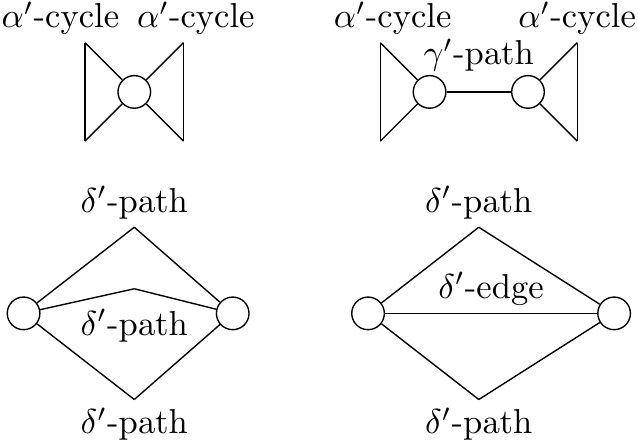}
    \subcaption{Basic graph}
    \label{fig:2_basic}
  \end{minipage}
  \begin{minipage}[b]{0.45\linewidth}
    \centering
    \includegraphics[scale=0.8]{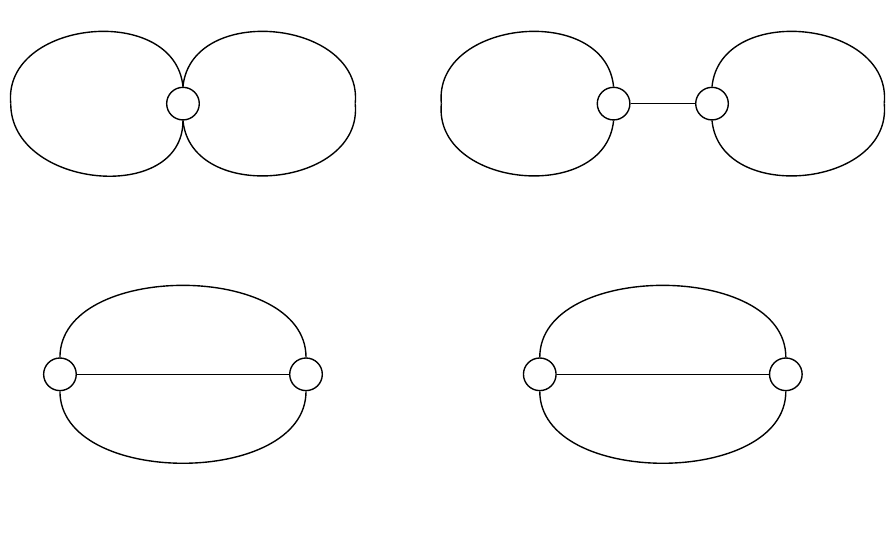}
    \subcaption{Multigraph}
    \label{fig:2_multi}
  \end{minipage}
  \caption{Example of the case $k=2$ (cf. \cite[Section 7]{W77}). The mapping $\psi$ transfers each of the basic graphs of  \ref{fig:2_basic} to each of the multigraphs of \ref{fig:2_multi}.}
  \label{fig:2}
\end{figure}

Recall that $N(n,k)$ is the number of connected $(n,n-1+k)$-graphs on $K_n$, which was introduced in Section 1. 
Let $W_k, k\ge1$ be the exponential generating function of $N(n,k)$:
$$
W_k(x)=\sum_{n=1}^{\infty}N(n,k)\frac{x^n}{n!}.
$$
Note that $W_1(x)$ is the exponential generating function for ``unicycles'' on $\{K_n\}_{n\ge1}$, which corresponds to $F_1(x,x)$ for bipartite graphs.

\begin{prop}[\cite{W77}]\label{prop:Wright} For $k\ge 1$, $W_k(x)$ is expressed by the summation with respect to basic graphs $\cA$:
\begin{align*}
W_k(x)=\sum_{\cA \in BG'_k}J_{\cA}(x)
\end{align*} 
with 
\begin{align}
J_{\cA}(x)=\frac{Y^{L'(\cA)}}{g_{\cA}(1-Y)^{M'(\cA)}},
\label{eq:J_A}
\end{align}
where $BG'_k$ is the set of basic graphs on complete graphs having \textit{$k$ cycles} and $M'(\cA) \le 3(k-1)$, $L'(\cA)$ and $g_{\cA}$ are constants depending only on $\cA$.
\end{prop}

\begin{lem} \label{lemma:coefJBM'}
For $\cA \in BG'_k$, there exist unique constants $\{\fa'_i(\cA)\}_{i=1}^{M'}, \{\fc'_j(\cA)\}_{j=0}^{L'-M'}$ such that
\begin{align*}
J_{\cA}=\sum_{i=1}^{M'} \frac{\fa'_i(\cA)}{(1-Y)^i}+\sum_{j=0}^{L'-M'}\fc'_j(\cA)Y^j,\ \ \fa'_{M'}(\cA)=\frac{1}{g_{\cA}}.
\end{align*}
\end{lem}
\begin{proof}
To show the second equation, put $\theta=1-Y$ in
 \eqref{eq:J_A} and apply the binomial expansion to the numerator. 
\end{proof}

For each $\cA \in BG'_k$ we contract their special cycles and paths and obtain a multigraph $\cA^*$. Define $\psi: BG'_k \to MG_k$ be the mapping of the contraction. Note that $\psi$ is not injective, but if $\psi(\cA_1)=\psi(\cA_2)=\cA^*$ for some $\cA_1, \cA_2$, then the difference of the two graphs is only due to the difference of their $\delta'$-paths and $\delta'$-edges. Define
$$
BG'_{k}|_{3(k-1)}:=\{\cA \in BG'_k : M'(\cA) = 3(k-1)\}.
$$
Let $\psi|_{3(k-1)}$ be the restriction to
$BG'_k|_{3(k-1)}$ of $\psi$, then this mapping is
bijective from $BG'_k|_{3(k-1)}$ to $MG_k$ i.e,
$\psi|_{3(k-1)}^{-1}(\cA^*)$ is a singleton. Indeed, if
$\cA^*$ has self-loops, replace them to
minimal $\alpha'$-cycles. Also, if $\cA^*$ has single edges
or multiple edges, replace them to $\gamma'$-paths or $\delta'$-paths, respectively. By this procedure, we obtain a unique basic graph 
$\cA \in BG'_k|_{3(k-1)}$, and then $\psi|_{3(k-1)}^{-1}(\cA^*)=\{\cA\}$. 

\subsection{Proof of the asymptotic equality \eqref{eq:f(n,n+k)asmthm}}\label{subsec:6.3}
We will prove the asymptotic equality \eqref{eq:f(n,n+k)asmthm}. Take $\cB^* \in MG_k$ such that there exists $\cB \in \phi^{-1}(\cB^*)$ satisfying $M(\cB)=3(k-1)$. Then, there exist unique $\cB_{{\rm id}} \in \phi^{-1}(\cB^*)$ and $\cA=\cA_{\cB^*}:=\psi^{-1}|_{3(k-1)}(\cB^*) \in BG'_k$. Then, we have
\begin{align}
g_{\cB_{{\rm id}}}=g_{\cA},
\label{eq:autoBidandA}
\end{align}
because mappings $\phi$ and $\psi$ preserve the connectivity between each of vertices in $\cB$ and $\cA$, respectively. By Proposition \ref{prop:coef1/(1-Y)^M} and \eqref{eq:autoBidandA}, we have
\begin{align}
\sum_{\substack{\cB \in \phi^{-1}(\cB^*) \\ M(\cB)=3(k-1)}} \fa_{3(k-1)}(\cB)
=\frac{1}{2^{k-1}}\frac{1}{g_{\cA}}.\label{eq:coefJ_BandJ_A}
\end{align}

From Proposition \ref{prop:Wright}, the asymptotic behavior of $\langle x^n \rangle W_{k-1}(x)$ is determined by the summation of $J_{\cA}(x)$ with respect to $\cA$ such that $M'(\cA)=3(k-1)$. Hence, by Lemma \ref{lemma:coefJBM'} we have
\begin{align*}
N(n,k)&=\langle x^n \rangle W_k(x) \\
&\sim \langle x^n \rangle\sum_{\cA \in BG'_k|_{3(k-1)}}J_{\cA}(x)\\ 
&\sim \sum_{\cA \in BG'_k|_{3(k-1)}}\langle x^n \rangle\left(\frac{\fa'_{3(k-1)}(\cA)}{(1-Y)^{3(k-1)}}\right)\\ 
&= \Bigg(\sum_{\cA \in BG'_k|_{3(k-1)}} \frac{1}{g_{\cA}}\Bigg)t_n(3(k-1)), \quad n \to \infty,
\end{align*}
where $t_n(p)$ is the tree polynomials defined by \eqref{eq:deftnp}.
On the other hand, 
by Lemma \ref{lemma:decomJB}, Proposition \ref{prop:coef1/(1+Y)^M}, \eqref{eq:coefJ_BandJ_A} and the fact that $\psi|_{3(k-1)}$ is bijective, we have
\begin{align*}
N_{\rm{bi}}(n,k)&=
\langle x^n \rangle F_{k}(x,x)\\
&\sim \langle x^n \rangle\sum_{\cB^* \in MG_{k}}\sum_{\substack{\cB \in \phi^{-1}(\cB^*) \\ M(\cB)=3(k-1)}}J_{\cB}(x,x)\\
&\sim\sum_{\cB^* \in MG_{k}}\sum_{\substack{\cB \in \phi^{-1}(\cB^*) \\ M(\cB)=3(k-1)}} 
\langle x^n \rangle \left(\frac{\fa_{3(k-1)}(\cB)}{(1-Y)^{3(k-1)}}\right)\\
&=\Bigg(\sum_{\cB^* \in MG_{k}}\sum_{\substack{\cB \in \phi^{-1}(\cB^*) \\ M(\cB)=3(k-1)}} \fa_{3(k-1)}(\cB) \Bigg)t_n(3(k-1))\\
&=\frac{1}{2^{k-1}}\Bigg(\sum_{\cA \in BG'_k|_{3(k-1)}} \frac{1}{g_{\cA}}\Bigg)t_n(3(k-1))\\
&\sim \frac{1}{2^{k-1}} N(n,k), \quad n \to \infty,
\end{align*}
hence the asymptotic equality \eqref{eq:f(n,n+k)asmthm} holds.

\section*{Acknowledgment}
This work was supported by JSPS KAKENHI Grant Numbers JP18H01124, JP20K20884 and JP23H01077, 
JSPS Grant-in-Aid for Transformative Research Areas (A) JP22H05105, 
and JST CREST Mathematics (15656429). 
TS was also supported in part by JSPS KAKENHI Grant Numbers, JP20H00119 and JP21H04432.

\end{document}